\documentclass[11pt]{article}
\usepackage{ifpdf}
\ifpdf
  \usepackage[pdftex]{graphicx}
  \usepackage{epstopdf}

\else
  \usepackage{graphicx}
  \usepackage{epsfig,psfrag}
  \psdraft\psfull
\fi
\usepackage{color,subfigure}
\usepackage{epic,eepic}
\usepackage{tabularx}
\usepackage{amsmath,amssymb,amsfonts}
\usepackage{amsthm}
\usepackage[margin=1in,top=0.75in,nohead]{geometry}
\usepackage{mathrsfs}
\usepackage[sub,ovp]{psfragx}
\usepackage{url}
\usepackage{setspace}
\usepackage{enumerate}
\usepackage{comment}
\usepackage{remark}

\usepackage{algpseudocode}
\usepackage{algorithm}
\newremark{remark}{Remark}

\usepackage{times}
\usepackage[utf8]{inputenc}
\usepackage{color}
\usepackage{amsmath}
\usepackage{times}
\usepackage{graphicx}
\usepackage{color}
\usepackage{multirow}
\usepackage{blindtext}
\usepackage{amsfonts}
\usepackage{amsmath}
\usepackage{times}
\usepackage{fancyhdr}
\usepackage{amssymb}
\usepackage{verbatim}
\usepackage{xspace}
\usepackage{cite}
\usepackage{comment}
\usepackage{lipsum}
\usepackage{wrapfig}

\usepackage{algorithm}
\usepackage[]{algpseudocode}

\usepackage{bbm}

\newtheorem{definition}{\bf{Definition}}
\newtheorem{theorem}{\bf{Theorem}}
\newtheorem{lemma}{\bf{Lemma}}

\newcommand{\rd}{{\mathrm d}}

\newcommand{\rp}{{\mathrm p}}

\newcommand{\rtr}{{\mathrm{tr}}}

\newcommand{\rT}{{\mathrm{T}}}

\newcommand{\vx}{{\bf x}}
\newcommand{\vy}{{\bf y}}
\newcommand{\vz}{{\bf z}}

\newcommand{\vq}{{\bf q}}

\newcommand{\vs}{{\bf  s}}

\newcommand{\vu}{{\bf u}}
\newcommand{\vv}{{\bf  v}}

\newcommand{\vh}{{\bf h}}

\newcommand{\vw}{{\bf w}}

\newcommand{\vF}{{\bf F}}

\newcommand{\vH}{{\bf H}}

\newcommand{\vQ}{{\bf Q}}
\newcommand{\vB}{{\bf B}}

\newcommand{\vP}{{\bf P}}
\newcommand{\vxi}{{\mbox{\boldmath$\xi$}}}

\newcommand{\vlambda}{{\mbox{\boldmath$\lambda$}}}

\newcommand{\vSigma}{{\mbox{\boldmath$\Sigma$}}}

\newcommand{\calH}{{\cal H}}
\newcommand{\calF}{{\cal F}}
\newcommand{\calB}{{\cal B}}
\newcommand{\calV}{{\cal V}}
\newcommand{\calP}{{\cal P}}
\newcommand{\calN}{{\cal N}}

\newcommand{\argmin}{\operatornamewithlimits{argmin}}

\newcommand{\veta}{{\bf \eta}}

\newcommand{\Pb}{\mathbb{P}}

\newcommand{\Rb}{\mathbb{R}}
\newcommand{\Eb}{\mathbb{E}}

\begin{document}

\title{Open-loop Deterministic Density Control of Marked Jump Diffusions}

\author{Kaivalya Bakshi and Evangelos A. Theodorou
\thanks{Kaivalya Bakshi is with the Southwest Research Institute, San Antonio, TX 78227 USA  (email: kaivalya.bakshi@swri.org)}% <-this % stops a space
\thanks{Evangelos A. Theodorou is with the Department of Aerospace Engineering, Georgia Institute of Technology, Atlanta,
	GA 30332 USA}}
	
\maketitle

%%%%%%%%%%%%%%%%%%%%%%%%%%%%%%%%%%%%%%%%%%%%%%%%%%%%%%%%%%%%%%%%%%%%%%%%%%%%%%%%%%%%%%%%%%%%%%%%%%%%%%%%%%%%%%%%%%%%%%%%%%%%%%%%%%%%%%%%%%%
\begin{abstract}

The standard practice in modeling dynamics and optimal control of a large population, ensemble, multi-agent system represented by it's continuum density, is to model individual decision making using local feedback information. In comparison to a closed-loop optimal control scheme, an open-loop strategy, in which a centralized controller broadcasts identical control signals to the ensemble of agents, mitigates the computational and infrastructure requirements for such systems. This work considers the open-loop, deterministic and optimal control synthesis for the density control of agents governed by marked jump diffusion  stochastic diffusion equations. The density evolves according to a forward-in-time Chapman-Kolmogorov partial integro-differential equation and the necessary optimality conditions are obtained using the infinite dimensional minimum principle (IDMP). We establish the relationship between the IDMP and the dynamic programming principle as well as the IDMP and stochastic dynamic programming for the synthesized controller. Using the linear Feynman-Kac lemma, a sampling-based algorithm to compute the control is presented and demonstrated for agent dynamics with non-affine and nonlinear drift as well as noise terms.
\end{abstract}

\section{Introduction}

Multi-agent systems consisting of large populations of identical agents with continuous time stochastic dynamics are used to model macro-economic systems, swarms in robotics and biology, vehicles in traffic and neuronal activity in  neuroscience. In the optimal control theory, approximating the system by a continuum density of the agents is a practical approach which makes the associated optimal control problem (OCP) mathematically tractable. Research topics treating the control of a distribution of agents to satisfy certain optimality and/or stability criterion fall into the category of density control. The standard idea is to model individual decision making as closed-loop controls by using individual local state-feedback. Since the individual states of agents are stochastic owing to random disturbances, the control actions are stochastic as well. Thus, it is possible to obtain a tractable mathematical model of a large population of agents that interact with each other in a decentralized, non-cooperative and closed-loop control setting. This framework corresponds to the stochastic optimal control theory \cite{YongBook_1958} in the (standard) case wherein the agents do not interact with each other and to the mean field game (MFG) theory \cite{Huang2007} in the (non-standard) case, wherein an agent's dynamics or control actions explicitly depend on the state of the other agents.

However, the assumption of local feedback control either (a) requires individual agents to compute the control locally in time and space (in the decentralized case) or (b) requires the centralized controller to transmit the unique optimal control signal to the respective individual (in the centralized case). The assumption of an open-loop deterministic control implies that the centralized controller can \textit{broadcast} (identical) feedforward control signals or feedback gains to be used by all the agents simultaneously, thus simplifying the loop besides mitigating design complexity and infrastructure requirements of the individual agent (\cite{Brockett2000}, \cite{Milutinovic2013}). Finally, since the application of deterministic state-parameterized gains will in general result in stochastic control actions for each individual although the optimization is posed on deterministic parameters, thus providing implicit feedback.

Further, with the exception of the linear-quadratic-Gaussian (LQG) regime, computation of the local feedback controls even for low dimensional agent dynamics is a non-trivial task. Moreover, if the agent dynamics are nonlinear in the state or controls or are excited by non-Gaussian noise, the numerical solution \textit{curse of dimensionality} \cite{Bellman1964} leads to intractability. Specifically, in the latter case of non-Gaussian excitation, the strategy of local (in space) linear-quadratic approximation \cite{Todorov2005} of the dynamics and cost functions cannot produce a scalable LQG-like algorithm. This is because in the non-Gaussian case, the optimality conditions are represented by the forward-in-time linear Fokker Planck (FP) PDE governing the density evolution and the backward-in-time semilinear Hamilton-Jacobi-Bellman (HJB) PDE governing the value equation which are coupled due to the appearance of the value function in the FP equation. The MFG case is further complicated due to the fully coupled nature of the HJB-FP system (\cite{grover2018mean}, \cite{Bakshi2018TCNS}, \cite{Bakshi2020TAC}). The first \cite{Yannis2016} and second (\cite{Bakshi_ACC2017}, \cite{Theodorou2019RSS}) order forward-backward SDE (FBSDE) \cite{YongBook_1958} framework has been applied to obtain algorithms for optimal control of dynamics with nonlinear drift and state multiplicative noise, but not in the case of control multiplicative Gaussian or the general case of non-Gaussian excitation \cite{Bakshi2015RSS}.

The main contribution of our work is the open-loop deterministic optimal control synthesis for systems with underlying non-Gaussian Q-marked Markov jump diffusion (QMJD) agent dynamics, where the drift and volatility terms are nonlinear in the state and controls. Such processes are used to model ecological population, financial and manufacturing processes (\cite{Pham2005} \cite{Hanson07appliedstochastic}). Prior works (\cite{Milutinovic2011}, \cite{Annunziato2015}, \cite{Berret2020}) on this topic are limited to systems with Gaussian noise. We formulate the associated OCP (section \ref{subsec:Problem_Formulation_Prelminaries}) and apply the infinite dimensional minimum principle (IDMP) to obtain the first order optimality conditions (section \ref{sec:Inf_Dim_Min_Principle}). Prior works applying the minimum principle for the OCP of open-loop deterministic density control were limited to the case of systems with Gaussian noise (\cite{Fattorini1999}, \cite{Krastanov2011}, \cite{Milutinovic2010}, \cite{Sanger11}). We also mention the works (\cite{Nisio2015}, \cite{Swiech2016}) which state a generalized formulation of the IDMP for the considered density control problem but do not obtain the optimality system for the case of Q-MJD processes which are the mainstay of our work.

The second contribution of our work is a sampling-based algorithm (section \ref{sec:SamplingBasedAlgorithms}) which iteratively samples the \textit{uncontrolled} dynamics to compute the control by evaluating the IDMP co-state over the state space. This approach has the key advantage of being able to benefit from \textit{na\"ive parallelization} \cite{Theodorou2011IPI} compared to the schemes in \cite{Annunziato2015} and \cite{Berret2020} which cannot be parallelized. The algorithm is demonstrated for density control of agents with nonlinear dynamics and non-Gaussian excitation by optimizing feedforward as well as state dependent (linear) control gains.

The third contribution of our work is to show the relationship between the IDMP and Dynamic Programming Priciple (DPP) applicable to the open-loop OCP of the density of Q-MJD processes (section \ref{sec:Relationship}). The relationship between the MP and DPP has attracted a lot of interested in the past and is well known in the deterministic case \cite{ConnectionStochasticZhou1991}. In the case of individual stochastic agents, the relationship of stochastic MP with DPP has been explored using the FBSDE formalism (\cite{Connection_Zhou1990}, \cite{ConnectionStochasticZhou1991}, \cite{ma1999forward}, \cite{Pham2014}) for systems with Gaussian and non-Gaussian (\cite{Framstad2004}. \cite{Shi2010}, \cite{Deshpande2014}, \cite{delong2013backward}) excitation. The IDMP-DPP relationship associated with the topic of the density control problem considered in this work has been addressed for the less general case of agents excited by Gaussian noise in the previous work \cite{Annunziato2015}. Our presentation is more general since it addressed the case of agents obeying Q-MJD processes. Additionally we comment on the relationship of the IDMP with stochastic dynamic programming.

\section{Problem Formulation}
\label{sec:Preliminaries}

In this section we provide assumptions and conditions related to existence and uniqueness of solutions to a general class of controlled QMJD processes. Two theorems describing the forward and backward Chapman-Kolmogorov PIDEs corresponding to evolution of the PDF representing QMJD processes are stated. The proofs of these theorems are given in the appendices.

\subsection{Definitions}
\label{subsec:Problem_Formulation_Prelminaries}

Let $(\Omega, \calB, \{\calB_t\}_{t \geq 0},\Pb)$ be a filtered probability space and $(\vx_t)_{t \geq 0}$ a process which is progressively measurable with respect to it. We follow \cite{Oksendal2009} and define this process over $\Rb^{n_x}$ described by
\begin{align}
\rd \vx_t =& \vF(t,\vx_t,\vu(t)) \rd t + \vB(t,\vx_t,\vu(t)) \rd \vw_t + \vH(t,\vx_t,\vQ) \rd \vP(t,\vx_t; t, \vQ) \notag \\
=& \vF(t,\vx_t,\vu(t)) \rd t + \vB(t,\vx_t,\vu(t)) \rd \vw_t + \int_{D_\vQ} \vH(t,\vx_t,\vq) \calP(t, \vx_t; \rd t,\rd \vq), \label{QMJDP}
\end{align}
\noindent where $\vx_t \in \Rb^{n_x}$, $\vu(t) \in D_\vu \subseteq \Rb^{n_u}$, $\vw_t \in \Rb^{n_w}$, $\vQ \in D_\vQ \subset \Rb^{n_p}$, $\vP \in \Rb^{n_p}$, $\vF:[0,T] \times \Rb^{n_x} \times \Rb^{n_u} \rightarrow \Rb^{n_x}$,  $\vB:[0,T] \times \Rb^{n_x} \times \Rb^{n_u} \rightarrow \Rb^{n_x \times n_w}$, $\vH:[0,T] \times \Rb^{n_x} \times D_Q  \rightarrow \Rb^{n_x \times n_p}$. The process $\vw_t$ is the standard Brownian motion. We denote by $\vH(t,\vx_t,\vq) = \vH(t,\vx_{t^-},\vq)$ in the above expressions so that $\Pi(t,\vx_t) = \vH(t,\vx_t,\vQ) \rd \vP(t,\vx_t; t, \vQ) = \int\limits_{0}^{t}\int\limits_{D_\vQ} \vH(t,\vx_t,\vq) \calP(t, \vx_t; \rd t,\rd \vq)$ under the zero-one law, is the doubly stochastic Poisson process \cite{Hanson2007}. Processes $\vx_t$, $\vw_t$, $\vP_t$ and functions on these processes are adapted to the considered filtration such that there exists a unique solution to this SDE given $\vx_0 = \vz \in \Rb^{n_x}$. The conditions for the existence and uniqueness of the solutions are provided in this subsection. Writing in matrix notation, vectors $\calP = [\calP_j]$ and $\vP = [P_j]$ are such that $\{\calP_j\}_{1 \leq j \leq n_p}$ are independent Poisson random measures and $\vQ$ is the mark vector with marks $\{Q_j\}_{1 \leq j \leq n_p}$ such that $Q_j \in D_{Q_j} \subset \Rb$, are independently distributed random variables independent of $P_j$. In this notation realizations of the mark random vector $\vQ$ in the Poisson random measure formulation are denoted by $\vq$. The advantage of this notation is that the mark vector has a deterministic representation. We notice that the processes $P_j$ conditioned on $\vx_t = \vx$ are Poisson distributed. We now assume that there exists the mark density function $\rp_{Q_j}$ corresponding to the mean measure $\nu_j$ of the Poisson random measure $\calP_j$ so that $\Eb[\calP_{j \omega}(t,\vx_t; \rd t,\rd q_j)|\vx_t = \vx] = \nu_j(\rd q_j) = \rp_{Q_j}(t, q_j; t, \vx) \lambda_j(t, q_j; t, \vx) \rd q_j \rd t $ where $\lambda_j \in \Rb$ is called the jump rate for the doubly stochastic Poisson process $P_j$. Since $\int_{D_{Q_j}} \rp_{Q_j}(t, q_j; t, \vx) \rd q_j = 1$ it is observed, writing in matrix vector notation, that
\small
\begin{align}
&\Eb[\rd \vP(t, \vQ; t, \vx_t)|\vx_t = \vx] = [\Eb[\rd P_j(t, Q_j; t, \vx_t)]|\vx_t = \vx] = \notag \\
& [\Eb[\int\limits_{D_{Q_j}} \calP_j(\rd t, \rd q_j; t, \vx_t)|\vx_t = \vx]] = [\Eb_{\rp_{Q_j}}[\lambda_j(t, Q_j; t,\vx) \rd t]]. \notag
\end{align}
\normalsize
Consequently for the case of mark independent jump rates $\lambda_j$ we would have $\Eb[\rd \vP(t,\vx_t; t, \vQ)|\vx_t = \vx] = [\lambda_j(t,\vx) \rd t] = \vlambda(t,\vx) \rd t$ which recovers the same result as in the simple Markov jump diffusions process. We denote by $\vh_{j}:[0,T] \times \Rb^{n_x} \times D_{Q_j} \rightarrow \Rb^{n_x}$ the $j^{th}$ column vector of the matrix $\vH(t,\vx_t,\vQ) = [h_{i,j}(t,\vx_t,Q_j)]$ as well as $\vSigma = \vB\vB^\rT$ and assume $h_{i,j}(t,\vx_t,\vQ) = \vh_{i,j}(t,\vx_t,Q_j)$. The process $\vx_t$ is referred to as the state variable and $\calV[t,T] \ni \vu:[t,T] \rightarrow \Rb^{n_u} $ as the control variable where $\calV[t,T] \triangleq \{ \vu(s) \in D_\vu| t \leq s \}$ fora ll $t \in [0,T]$ is the class of optimal controls. Here the symbol $\triangleq$ implies definition. We comment on this class of controls later in section \ref{sec:Preliminaries.problemstatement}. 

Let us make the following assumptions for the above defined controlled stochastic process:
\vspace{-3mm}
\begin{itemize}
	\item [(S1)] there exists a constant $C_1<\infty$ such that for all $t \in [0,T]$, for all $\vx \in \Rb^{n_x}, \vu \in D_\vu$
	\begin{align}
	&||\vB(t,\vx,\vu)||^2 + |\vF(t,\vx,\vu)|^2 + \sum\limits_{k = 1}^{n_p} \int_{D_{Q_j}} |\vh_{j}(t,\vx)|^2 \nu_j(\rd q_j) \leq C_1(1 + |\vx|^2 + |\vu|^2) \notag
	\end{align}
	\item [(S2)] there exists a constant $C_2 < \infty$ such that for all $t \in [0,T]$,  for all $\vx , \hat{\vx}\in \Rb^{n_x} \; \text{and} \; \vu, \hat{\vu} \in \Rb^{n_u}$
	\begin{align}
	&||\vB(t,\vx,\vu) - \vB(t,\hat{\vx},\hat{\vu})||^2 + |\vF(t,\vx,\vu) - \vF(t,\hat{\vx},\hat{\vu})|^2 \notag \\
	& + \sum\limits_{j = 1}^{n_p} \int_{D_{Q_j}} |\vh_{j}(t,\vx,q_j) - \vh_{j}(t,\hat{\vx},q_j)|^2 \nu_j(\rd q_j) \leq C_2(|\vx - \hat{\vx}|^2 + |\vu - \hat{\vu}|^2) \notag
	\end{align}
	
	\item [(S3)] $F_i(t,\vx,\vu)$ is once continuously differentiable w.r.t. $\vx$ for all $i$
	\item [(S4)] $\Sigma_{ij}(t,\vx,\vu)$ is twice continuously differentiable w.r.t. $\vx$ for all $i,j$
	\item [(S5)]  $\vh_{j}:[0,T] \times \Rb^{n_x} \times D_{Q_j} \rightarrow \Rb^{n_x}$ is a bijection from $\Rb^{n_x}$ to $\Rb^{n_x}$, for all $t \in [0,T]$, for all $q_j \in D_{Q_j}$ and $\vh_{j}(t,\vx,q_j) = \veta_{j}(t,\vxi_j,q_j)$, $I - \veta_{j \vxi_j}(t, \vxi_j, q_j) \neq \mathbf{0}$ for all $(t,\vxi_j)$ where $\vxi_j = \vx + \vh_{j}(t,\vx, q_j)$
	\item [(S6)] $F_i(t,\vx,\vu)$, $\Sigma_{ij}(t,\vx,\vu)$ is once continuously differentiable w.r.t. $\vu \in D_\vu$ for all $i$
\end{itemize}
\vspace{-3mm}
Under the assumptions (S1), (S2), it is well known (pp 10, theorem 1.19) \cite{Oksendal2009} that \eqref{QMJDP} admits a unique c\'adl\'ag adapted integrable solution $\vx_t$ for all $\in [0,T]$ given $\vx_0 = \vz \in \Rb^{n_x}$. The controlled stochastic process represented by the SDE \eqref{QMJDP} is variously called as the marked, compound or doubly stochastic jump diffusion process. Assumptions (S3) through (S5) are typical differentiability assumptions for the existence of the forward and backward Chapman-Kolmogorov operators and corresponding PIDEs.
\begin{definition} \textit{Assuming (S5) we define the forward Chapman-Kolmogorov operator that corresponds to QMJD process \eqref{QMJDP}, denoted by $\calF^\vu_{\text{MJD}}(\cdot)$ acting on a function $(\cdot):[0,T] \times \mathbb{R}^{n_x} \rightarrow \mathbb{R}$, given the control $\vu$, when it exists, as}
	\begin{align}\label{Forward_CKOperator}
	\calF^\vu_{\text{\tiny MJD}} \; (\cdot)(t,\vx)  =& \sum\limits_{i = 1}^{n_x} -\frac{\partial}{\partial x_i} ((\cdot)(t,\vx) F_i(t,\vx,\vu)) + \frac{1}{2}  \sum\limits_{i,j = 1}^{n_x}\frac{\partial^2}{\partial x_i  \partial x_j} ( \Sigma_{ij}(t,\vx,\vu) (\cdot)(t,\vx) ) \notag \\
	&+ \sum\limits_{j = 1}^{n_p} \int\limits_{D_{Q_j}} \bigg( (\cdot)(t,\vx - \veta_{j}(t,\vx,q_j))  |I - \veta_{j \vx}(t, \vx, q_j)| - (\cdot)(t,\vx) \bigg)\rp_{Q_j} \lambda_j (t, \vx; t, q_j) \rd q_j.
	\end{align}
	\end{definition}
%\noindent The operator $\mathcal{F}^\vu_\text{\tiny MJD}$, given the control $\vu$ is also called the Kolmogorov Feller operator \cite{Danilov2013}.
%\begin{definition} \textit{Consider the Markov Jump Diffusion Process  in  (\ref{MJD}). We define the backward Chapman-Kolmogorov operator that corresponds to the MJD process, acting on a function $(\cdot):[0,T] \times \mathbb{R}^{n_x} \rightarrow \mathbb{R}$ given the control $\vu$, when it exists, as}
% 	\begin{align} \label{Backward_CKOperator}
% 	&\mathcal{F}^{\dagger \; \vu}_{\text{\tiny MJD}} (\cdot)(t,\vx) = \sum\limits_{i= 1}^{n_x}  F_i(t,\vx,\vu) \frac{\partial}{\partial x_i} (\cdot)(t,\vx) \notag \\
% 	&+ \frac{1}{2}  \sum\limits_{i,j = 1}^{n_x} [\vSigma(t,\vx,\vu)]_{ij} \frac{\partial^2}{\partial x_i \partial x_j} (\cdot)(t,\vx) + \rd_\text{jump} (\cdot) (t,\vx)
% 	\end{align}
% 	\noindent \textit{where} $\rd_\text{jump} (\cdot)(t,\vx) \triangleq \sum\limits_{j = 1}^{n_p} \text{Jump}_j (\cdot)(t,\vx) \lambda_j(t) \triangleq  \sum\limits_{j = 1}^{n_p} \Big[ (\cdot)(t, \vx + \vh_{j}(t,\vx)) - (\cdot)(t,\vx) \Big] \lambda_j(t)$.
% \end{definition}

\begin{definition} \textit{We define the backward Chapman-Kolmogorov operator that corresponds to QMJD process \eqref{QMJDP}, denoted by $\mathcal{F}^{\dagger \; \vu}_{\text{\tiny QMJD}}(\cdot)$ acting on a function $(\cdot):[0,T] \times \mathbb{R}^{n_x} \rightarrow \mathbb{R}$, given the control $\vu$, when it exists, as}
	\begin{align} \label{Backward_CKOperator}
	\mathcal{F}^{\dagger \; \vu}_{\text{\tiny MJD}} (\cdot)(t,\vx) =& \sum\limits_{i= 1}^{n_x}  F_i(t,\vx,\vu) \frac{\partial (\cdot)}{\partial x_i} + \frac{1}{2}  \sum\limits_{i,j = 1}^{n_x} [\vSigma(t,\vx,\vu)]_{ij} \frac{\partial^2 (\cdot)}{\partial x_i \partial x_j} + \rd_\text{jump} (\cdot),
	%&= \vF^\rT \nabla_{\vx} \pi(\vx,t) + \frac{1}{2} \rtr(\nabla_{\vx \vx}\pi(\vx,t) \vB \vB^\rT) + \sum\limits_{j = 1}^{n_p} Jump_j \pi(\vx,t) \lambda_j(t) \notag \\
	%&= \vF^\rT \nabla_{\vx} \pi(\vx,t) + \frac{1}{2} \rtr(\nabla_{\vx \vx}\pi(\vx,t) \vB \vB^\rT) + d_{jump} \pi(\vx,t) \label{4.11}
	\end{align}
	\noindent \textit{ where} $\text{Jump}_j (\cdot)(t,\vx,q_j) \triangleq (\cdot)(t, \vx + \vh_{i,j}(t,\vx,q_j)) - (\cdot)(t,\vx)$ \textit{and}
	\begin{align}
	\rd_\text{jump} (\cdot)(t,\vx) \triangleq \sum\limits_{j = 1}^{n_p} \int\limits_{D_{Q_j}} \text{Jump}_j (\cdot)(t,\vx, q_j) \rp_{Q_j} \lambda_j (t, q_j; t, \vx) \rd q_j. \notag
%	=&  \sum\limits_{j = 1}^{n_p} \int\limits_{D_{Q_j}} \Big[ (\cdot)(t, \vx + \vh_{i,j}(t,\vx,q_j)) - (\cdot)(t,\vx) \Big] \rp_{Q_j} \lambda_j (t, \vx; t, q_j) \rd q_j. \notag
	\end{align}
\end{definition}

\subsection{Forward and Backward Chapman-Komogorov PIDEs}

%{Our objectives in this paper pertain to the development of algorithms for PDF control and theoretical exposition of the relationship of the IDMP and DPP for PDF control. Therefore we assume regularity conditions on the PDF of the process as well as the costate function when applying the IDMP. This also serves well our goal of developing control algorithms.}

%The forward and backward Chapman-Kolmogorov PIDEs are the fundamental describing equations \cite{Gardiner} for the PDF of the general class of MJD processes defined above. {The forward PIDE describes the time evolution of the PDF of all possible ensembles governed by SDE \eqref{QMJDP}}.
We do not discuss the existence and uniqueness of solutions to the forward Chapman-Kolmogorov PIDE here. Instead we refer the interested reader to \cite{Garroni1992} for properties of continuous solutions to the forward Chapman-Kolmogorov PIDEs. Instead we follow the approach of \cite{Bect2010} which only derives the forward PIDE that the PDF should satisfy given that the PDF exists and has certain smoothness. Derivations of the forward Chapman-Kolmogorov PIDE is a well explored topic for the case of spontaneous and forced jumps \cite{Bect2010}, \cite{Gardiner}, \cite{Hespanha2005}. An explicit derivation in case of simple jump diffusions, called the differential Chapman-Kolmogorov PIDE can be seen in (pp 51, equation 3.4.22) \cite{Gardiner}.

In this work we follow the approach of Hanson \cite{Hanson07appliedstochastic}. Here we give complete proof of the second part of the result stated (pp 203-204, Theorem 7.7) \cite{Hanson07appliedstochastic} in theorem \ref{Thm1} below. This is the multidimensional version of the more detailed one dimensional result (pp 199-202, Theorem 7.5) \cite{Hanson07appliedstochastic}. Theorem \ref{Thm1} states the additional smoothness conditions for the PDF besides (S1), (S2), (S3), (S4), (S5) under which the forward Chapman-Kolmogorov equation is satisfied by the PDF of multidimensional QMJD processes \eqref{QMJDP}. Theorem \ref{Thm2} shows how the backward and forward Chapman-Kolmogorov operators are formal adjoints of each other under certain conditions. Proofs are presented in the appendix \ref{Appendix} for completeness. Define $\mathcal{L}^2$ inner product of functions $f_1,f_2$
{
\begin{equation}\label{eq:inner_product}
 \bigg<f_1,f_2 \bigg>  =  \int f_{1}(\vx) f_{2}(\vx) \rd \vx.
\end{equation}
}
\begin{theorem} \label{Thm1}
	\noindent \textit{Consider the QMJD process in \eqref{QMJDP} such that assumptions (S1), (S2), (S3) and (S4) in section \ref{subsec:Problem_Formulation_Prelminaries} are true. Let there exist $\rp(t,\vx|\tau, \vy_\tau)$, the transition probability density of the state $\vx_t$ for all $t \in [0,T]$ written in short as $\rp(t,\vx)$. If (S5) is true,
		\begin{itemize}
%			\item[(S5)] is true.
			\item[(T1)] there exists $v:\Rb^{n_x} \rightarrow \Rb$  a bounded arbitrary test function which is twice differentiable w.r.t. $\vx$ such that for $\vu \in D_\vu$ the conjunct below vanishes
			\begin{align}
			\sum\limits_{i,j = 1}^{n_x}  \int\limits_{\Rb^{n_x}} \frac{\partial }{\partial x_i} \Bigg( F_i(t,\vx,\vu) \rp(t,\vx) v(\vx) &- \frac{1}{2} \frac{\partial}{\partial x_j} \Big( \Sigma_{ij} (t,\vx,\vu) \rp(t,\vx) \Big) v(\vx) \notag \\ 
			&\qquad\qquad+ \frac{1}{2} \Sigma_{ij}(t,\vx,\vu) \rp (t,\vx)
			\frac{\partial}{\partial x_j} v(\vx) \Bigg) \rd\vx = 0, \notag \\
			\end{align}
			\item[(T2)] $\rp(t,\vx)$ is once continuously differentiable w.r.t. $t$, $(\rp F_i)(t,\vx,\vu)$ is once continuously differentiable w.r.t. $\vx$ for all $i$,  $(\Sigma_{ij}\rp)(t,\vx,\vu)$ is twice continuously differentiable w.r.t. $\vx$ for all $i,j$, $\vu \in D_\vu$,
		\end{itemize}
		\noindent then $\rp(t,\vx)$ satisfies the forward Chapman-Kolmogorov PIDE}
	\begin{align}\label{Forward_CKPDE}
	\frac{\partial \rp(t,\vx)}{\partial t} = \mathcal{F}^{\vu}_\text{\tiny MJD} \; \rp(t,\vx).
	\end{align}
	\noindent \textit{in the weak sense. Further, given $\vx_{t_0} = \vx_0$, the PDF $\rp(t_0,\vx)$ satisfies the delta function initial condition}
	\begin{equation} \label{PDF_canbe_delta}
	\lim\limits_{t \downarrow t_0}\rp(t,\vx) = \delta(\vx - \vx_0).
	\end{equation}
	%	 Further, $\forall \vu \in D_\vu$ such that the left and right hand sides of equation \eqref{4.2} are well defined, $\pi(t,\vx) $ satisfies the Green's identity \cite{Lanczos1961}}
	%	\begin{equation}
	%	\bigg< \pi(t,\vx), \mathcal{F}^{\vu}_\text{\tiny MJD} \rp(t,\vx) \bigg> =  \bigg< \rp(t,\vx),\mathcal{F}^{\dagger \; \vu}_\text{\tiny MJD} \pi (t,\vx) \bigg>. \label{4.2}
	%	\end{equation}
	%	\textit{We then call $\pi$ the adjoint function to $\rp$ so that $ \mathcal{F}^{\dagger \; \vu}_\text{\tiny MJD} (\cdot)$ is the adjoint operator of $\mathcal{F}^\vu_\text{\tiny MJD} (\cdot)$.}
\end{theorem}

\begin{theorem} \label{Thm2}
	\textit{Consider the QMJD process  in  (\ref{QMJDP}) such that assumptions (S1) through (S5) in subsection \ref{subsec:Problem_Formulation_Prelminaries} are true. We assume that conditions (T1) and (T2) stated in theorem \ref{Thm1} are true.} 
	\textit{If}
	\begin{itemize}
		\item[(T3)] \textit{$\pi:[0,T] \times \Rb^{n_x} \rightarrow \Rb$ is a bounded function which is twice differentiable w.r.t. $\vx$ and once continuously differentiable w.r.t. $t$ such that $\forall \vu \in D_\vu$
			\begin{align}
			\sum\limits_{i,j = 1}^{n_x}  \int\limits_{\Rb^{n_x}} \frac{\partial }{\partial x_i} \Bigg( F_i(t,\vx,\vu) \rp(t,\vx) \pi(t, \vx) &- \frac{1}{2} \frac{\partial}{\partial x_j} \Big( \Sigma_{ij} (t,\vx,\vu) \rp(t,\vx) \Big) \pi(t,\vx) \notag \\ 
			&\qquad + \frac{1}{2} \Sigma_{ij}(t,\vx,\vu) \rp (t,\vx)
			\frac{\partial}{\partial x_j} \pi(t,\vx) \Bigg) \rd\vx = 0, \notag \\
			\end{align}
			\item[(T4)] the left and right hand sides of equation \eqref{4.2} are bounded,}
	\end{itemize}
	\textit{then for all $\vu \in D_\vu$ $\pi(t,\vx)$ satisfies the Green's identity \cite{Lanczos1961} for all $\vu \in D_\vu$}
	\begin{equation}
	\bigg< \pi(t,\vx), \mathcal{F}^{\vu}_\text{\tiny MJD} \rp(t,\vx) \bigg> =  \bigg< \rp(t,\vx),\mathcal{F}^{\dagger \; \vu}_\text{\tiny MJD} \pi (t,\vx) \bigg>. \label{4.2}
	\end{equation}
	\textit{We then call $\pi$ the adjoint function to the PDF $\rp$ so that $\mathcal{F}^{\dagger \; \vu}_\text{\tiny MJD} (\cdot)$ is the adjoint operator of $\mathcal{F}^\vu_\text{\tiny MJD} (\cdot)$}.
\end{theorem}

\subsection{Problem Statement}
\label{sec:Preliminaries.problemstatement}

Using the inner product definition \eqref{eq:inner_product}, we may define the cost functional, when it exists, as
\vspace{-3mm}
\begin{align}\label{CostFunctional}
J \left(t;\rp,\vu\right) = \Phi(T;\rp(T,\vx)) +  \int\limits_{t}^{T} \mathcal{L}(s,\vu(s); \rp(s,\vx)) \; \rd t,
\end{align}
where $\Phi(T;\rp(T,\vx))$ is called the expected terminal cost functional,  $\mathcal{L}(t,\vu(t);\rp(t,\vx))$ is called the expected running cost functional and $t \in [0,T)$. Define
\begin{align}
\Phi(T; \rp(T,\vx))  = \bigg< \phi(T,\vx),\rp(T,\vx) \bigg> \; \text{and} \; \mathcal{L}(s, \vu(s); \rp(s,\vx)) = \bigg<\ell (s,\vx,\vu(s)),\rp(s,\vx) \bigg>,
\end{align}
\noindent wherein $ \phi: [0,T] \times \Rb^{n_x} \to \Rb$ is called the terminal cost  function and $\ell:[0,T] \times \Rb^{n_x} \times \Rb^{n_{u}} \to \Rb$ the running cost function. The choice of these functions as well as the class of admissible control functions $\calV[t,T]$ is restricted such that the $\phi$ and $\ell$ are integrable at all times $t \in [0,T]$. The infinite dimensional optimal control problem for the QMJD processes \eqref{QMJDP} is then stated as
\begin{align}\label{Problem}
& \underset{\vu \in \calV[t,T]}{\text{min}} J \left(t;\rp,\vu\right),
\end{align}
\noindent subject  to the dynamics 
\begin{equation} 
\frac{\partial \rp(s,\vx)}{\partial s} = \mathcal{F}^{\vu(s)}_\text{\tiny MJD} \rp(s,\vx), \; \rp(t,\vx) = \rp_0(\vx). \label{5.3}
\end{equation}
%\noindent For the remainder of this paper, referring to equation \eqref{Problem} assumes the dynamics \eqref{5.3}.
\noindent  Henceforth the stochastic control problem \eqref{Problem} subject to the dynamics \eqref{5.3}, is referred to simply as problem \eqref{Problem}.
%Notice that when $t = 0$ we must have from theorem \ref{Thm1} that the initial PDF must be $\rp_0(\vx) = \delta(\vx - \vz)$ where $\vz \in \Rb^{n_x}$ such that the PIDE \eqref{Forward_CKPDE} represents the PDF of the MJDP \eqref{QMJDP}.
Put in words, the problem undertaken is to find a \textit{deterministic open loop} optimal control for all $s \in [t,T]$ to minimize the cost \eqref{CostFunctional} over $[t,T]$ given the PDE dynamics for $\rp(s,\vx)$. {The solution is a \textit{broadcast} controller for all the stochastic ensembles governed by \eqref{QMJDP} which solves the problem \eqref{Problem}}.

\section{Infinite Dimensional Minimum Principle for Q-marked Jump Diffusions}
\label{sec:Inf_Dim_Min_Principle}

This section contains a detailed application of the IDMP as applied to the SOC problem \ref{Problem}. First a Hamiltonian functional for the IDMP is defined. We then derive the Euler-Lagrange equations representing the necessary optimality conditions for the formulated problem. We conclude by defining an infinite dimensional optimal control for PIDE control of QMJD processes.
\begin{definition}
	\textit{We define the Hamiltonian functional for the infinite dimensional MP given the control $\vu$, when it exists, by}
	\begin{equation}
	\calH\big(s, \vu; \rp(s,\vx), \pi(s,\vx) \big)= \bigg<\ell(s,\vx, \vu), \rp (s,\vx)\bigg> +  \bigg< \mathcal{F}^{\dagger \; \vu}_\text{\tiny MJD} \pi(s,\vx),  \rp(s,\vx) \bigg>, \label{Hamiltonian_IDMP}
	\end{equation}      	
	\noindent \textit{where $\pi:[0,T] \times \Rb^{n_x} \rightarrow \Rb$ is called the IDMP costate function, $\ell$ is the running cost function and $\rp$ is the probability density function representing the MJD process \eqref{QMJDP}.}
\end{definition}
\begin{theorem} (\textbf{Infinite Dimensional Minimum Principle}) \label{Thm3}
	\textit{Consider the Markov Jump Diffusion Process  in  \eqref{QMJDP} such that assumptions (S1), (S2), (S3), (S4), (S5), (S6) in section \ref{subsec:Problem_Formulation_Prelminaries} are true. We assume that conditions (T1) and (T2) stated in the theorem \ref{Thm1} hold true. 
		%		We assume the necessary conditions required \cite{Bect2010} for the existence of unique probability density and adjoint functions $\rp, \; \pi$ satisfying the forward and backward Chapman-Kolmogorov PDEs respectively on the domain $[0,T] \times D_{\vx}$.
		%	Consider the infinite dimensional optimal control problem \eqref{Problem} with the cost function under minimization as specified subject to the dynamics of the forward Chapman-Kolmogorov PIDE \eqref{5.3}. 
	and  that {the running cost} $\ell$ is once continuously differentiable w.r.t. $\vu$.}
	%	Let the Hamiltonian functional for the infinite dimensional minimum principle be defined as}
	%	\begin{align}
	%	&\calH^{{\text{\tiny  IDMP}}}\big(t; \rp(t,\vx), \vu_t, \pi(t,\vx) \big) \notag \\
	%	&\triangleq \mathcal{L}\bigg(t;\rp(t,\vx),\vu_t \bigg)+ \bigg<\pi(t,\vx), \frac{\partial \rp(t,\vx)}{ \partial t}\bigg> \notag \\
	%	&= \bigg<\ell(t,\vx), \rp (t,\vx)\bigg> +  \bigg< \pi(t,\vx), \mathcal{F}_{\text{MJD}}  \rp(t,\vx) \bigg> \label{5.7}
	%	\end{align}      	
	%	\noindent \textit{where $\pi(t,\vx)$ is called the Lagrange multiplier function or the Costate function.
	{\textit{Furthermore we assume that there exists a function $\pi$ called the IDMP costate function such that the Hamiltonian functional for the infinite dimensional MP, $\calH\big(s;\vu(s), \rp(s,\vx), \pi(s,\vx) \big)$, exists for this choice and is Frechet differentiable w.r.t. $\vu$. If the IDMP costate function $\pi$ satisfies conditions (T3) and (T4) stated in theorem \ref{Thm2} then the necessary conditions for optimality on the domain $[t,T]$ for the infinite dimensional optimal control problem \eqref{Problem} subject to the dynamics of the forward Chapman-Kolmogorov PIDE \eqref{5.3}, are the Euler-Lagrange equations and terminal condition:}
		\begin{align}
		&\calH_{\vu}(s,\vu(s);\rp(s,\vx),\pi(s,\vx)) = 0 \label{Euler_Lagrange_1} \\
		&-\frac{\partial \pi(s,\vx)}{\partial s} = \ell(s,\vx, \vu(s)) + \mathcal{F}^{\dagger \; \vu(s)}_{\text{\tiny MJD}} \pi(s,\vx) \label{Euler_Lagrange_2} \\
		&\pi(T,\vx) = \phi(T,\vx). \label{Euler_Lagrange_3}
		\end{align}
		\textit{Further the Frechet derivative of the Hamiltonian functional $\calH$ w.r.t. $\vu$ can be specified as}
		\begin{align}
		&\calH_\vu(s,\vu(s);\rp(s,\vx),\pi(s,\vx)) \notag \\
		&= \bigg<\ell_\vu(s,\vx,\vu) + \sum\limits_{i = 1}^{n_x}\frac{\partial F_i}{\partial \vu}(s,\vx,\vu(t)) \frac{\partial \pi(s,\vx)}{\partial x_i} + \frac{1}{2}\sum\limits_{i,j = 1}^{n_x}\frac{\partial}{\partial \vu} \bigg( \frac{\partial^2  (\Sigma(s,\vx,\vu(s))_{ij}\pi(s,\vx))}{\partial x_i \partial x_j} \bigg),\rp(s,\vx) \bigg>. \label{Hamlitonian_IDMP_Control_Gradient}
		\end{align}}
\end{theorem}
\begin{proof}	
	\normalfont
	Slight abuse of notation is employed in this proof by neglecting to write function arguments for brevity.  In the spirit of applying the infinite dimensional minimum principle we append the dynamics into the cost by introducing the Lagrange multiplier or IDMP costate. We write the auxiliary cost functional
	\vspace{-3mm}
	\begin{align}
	&J^\star\left(t; \rp,\vu \right) = \bigg<\phi(T,\vx),\rp(T,\vx)\bigg> + \int\limits_{t}^{T} \; \bigg[ \bigg<\ell,  \rp\bigg> + 
	\bigg<  \pi, \bigg( \mathcal{F}^{\vu}_{\text{\tiny MJD}} \rp - \frac{\partial \rp}{\partial s} \bigg) \bigg> \bigg] \; \rd s \notag \\  
	=& \bigg<\phi(T,\vx),\rp(T,\vx)\bigg>  + \int\limits_{t}^{T} \bigg[ \; \bigg<\ell,  \rp\bigg>  + \bigg< \rp,\mathcal{F}^{\tiny \dagger \; \vu}_{\text{\tiny MJD}} \pi \bigg> -  \bigg<\pi,\frac{\partial \rp}{\partial s} \bigg> \bigg] \; \rd s, \label{5.8}
	\end{align}	
	\noindent since $\mathcal{F}^{\dagger \; \vu}_{\text{\tiny MJD}}(\cdot)$ satisfies $\big< \pi,\mathcal{F}^\vu_{\text{\tiny MJD}} \rp \big> = \big<\rp,\mathcal{F}^{\dagger \;\vu}_{\text{\tiny MJD}} \pi \big>$ since necessary conditions for theorem \ref{Thm2} are assumed to be true here. We then have
	\begin{align}
	J^\star\left(t; \rp,\vu \right) &= \bigg<\phi(T,\vx),\rp(T,\vx)\bigg>+ \int\limits_{t}^{T} \bigg[ \bigg< \ell + \mathcal{F}^{\dagger \; \vu}_{\text{\tiny MJD}} \pi,\rp\bigg> - \bigg<\pi,\frac{\partial \rp}{\partial s}\bigg> \bigg]  \rd s \notag \\
	&= \bigg<\phi(T,\vx),\rp(T,\vx)\bigg> + \int\limits_{t}^{T} \bigg[ \calH(s,\vu;\rp,\pi)  - \bigg<\pi,\frac{\partial \rp}{\partial s}\bigg> \bigg] \rd s, \label{5.9}
	\end{align}
	\noindent from the definition of Hamiltonian for IDMP \eqref{Hamiltonian_IDMP}. As the Hamiltonian functional is Frechet differentiable,
	\begin{align}
	& J^\star \left(t;\rp + \delta \rp, \vu + \delta \vu \right) = \bigg<\phi(T,\vx),\rp(T,\vx) + \delta \rp(T,\vx)\bigg> \notag \\
	&+ \int\limits_{t}^{T} \; \bigg[ \bigg<\ell + \mathcal{F}^{\dagger \; \vu}_{\text{\tiny MJD}} \pi + (\ell_{\vu} + (\mathcal{F}^{\dagger \; \vu}_{\text{\tiny MJD}} \pi)_{\vu})^\rT \delta \vu,\rp + \delta \rp \bigg> - \bigg<\pi,\frac{\partial }{\partial s} (\rp + \delta \rp)\bigg> \bigg] \; \rd s, \label{5.10}
	\end{align}	
	\noindent on neglecting higher order terms of the variations of $\rp$ and $\vu$ so that
	%	 since we need to obtain only the first order variation of the functional $J^*$ with respect to $\rp$, $\vu$.
	\vspace{-3mm}
	\begin{align} 
	& J^\star \left(t;\rp + \delta \rp, \vu + \delta \vu \right)= J^\star \left(\rp, \vu \right) + \int\limits_{t}^{T} \; \bigg[ \bigg<\ell + \mathcal{F}^{\dagger \; \vu}_{\text{\tiny MJD}} \pi,\delta \rp\bigg> + \bigg< \bigg(\ell + \mathcal{F}^{\dagger \; \vu}_{\text{\tiny MJD}} \pi \bigg)_{\vu}^\rT \delta \vu,\rp\bigg> \bigg]\; \rd s, \label{5.11}
	\end{align}	
	\noindent on neglecting higher order mixed variational terms of $\rp$ and $\vu$.  Therefore the first order variation of the auxiliary cost functional $J^\star$ with respect to the functions $\rp(s,\vx)$, $\vu(s)$ is given by	
	\begin{align}
	&\delta J^\star \left(t; \rp, \vu \right) = \bigg<\phi(T,\vx),\delta \rp(T,\vx)\bigg> + \int\limits_{t}^{T} \bigg[ \; \bigg<\ell + \mathcal{F}^{\dagger \; \vu}_{\text{\tiny MJD}} \pi,\delta\rp\bigg> \nonumber \\
	&+ \bigg<\bigg(\ell_{\vu}+ (\mathcal{F}^{\dagger \; \vu}_{\text{\tiny MJD}} \pi )_{\vu} \bigg)^\rT \; \delta \vu,\rp\bigg>  - \bigg<\pi,\frac{\partial }{\partial s} (\delta \rp)\bigg> \bigg] \; \rd t. \label{5.12}	
	\end{align}	
	\noindent Equations $\eqref{Hamiltonian_IDMP}$, $\eqref{5.8}$ under assumption of Frechet differentiability of the Hamiltonian functional imply	
	\begin{align}
	\bigg< \bigg(\ell + \mathcal{F}^{\dagger \; \vu}_\text{\tiny MJD} \pi\bigg)_{\vu}^\rT \delta \vu,\rp\bigg> &= \bigg<\bigg(\ell + \mathcal{F}^{\dagger \; \vu}_\text{\tiny MJD}\bigg)_{\vu}^\rT,\rp \bigg> \delta \vu(s) = \calH^{\rT}_\vu(s,\vu;\rp,\pi) \; \delta \vu(s). \label{5.13}
	\end{align}	
	\noindent The last term inside the integral in the RHS of $\eqref{5.12}$ can now be integrated by parts over the time $t$. Changing the order of integration is permitted since conditions of Fubini's theorem \cite{ThomasFinney_1996}, namely the relevant continuous differentiability conditions are satisfied. Therefore
	\begin{align}
	&\int\limits_{t}^{T} \int\limits_{\Rb^{n_x}} \pi(s,\vx) \frac{\partial \delta \rp} {\partial t}(s,\vx) \; \rd \vx \; \rd s = \int\limits_{\Rb^{n_x}} \int\limits_{t}^{T} \pi(s,\vx) \frac{\partial \delta \rp}{\partial s} (s,\vx) \; \rd s \; \rd \vx  \notag \\
%	&= \int\limits_{\Rb^{n_x}} \left[ [\pi(s,\vx) \delta\rp(s,\vx)] \Big|_{t}^{T} - \int\limits_{t}^{T} \; \frac{\partial \pi(s,\vx)}{\partial s} \delta \rp (s,\vx) \rd s \right] \rd \vx \notag \\
%	&= \int\limits_{\Rb^{n_x}} \; [\pi(s,\vx){\delta \rp (s,\vx)}]\Big|_{t}^{T} \; \rd \vx  - \int\limits_{t}^{T} \; \bigg<\frac{\partial \pi(s,\vx)}{\partial s}, \delta \rp (s,\vx) \bigg> \; \rd s  \notag \\
	&= \bigg<\pi(T,\vx),{\delta\rp(T,\vx)}\bigg> - \bigg<\pi(t,\vx),{\delta \rp(t,\vx)}\bigg>  - \int\limits_{t}^{T} \; \bigg<\frac{\partial \pi(s,\vx)}{\partial s}, \delta \rp (s,\vx) \bigg> \; \rd s, \label{5.14}
	\end{align}	
	\noindent where we note that $\delta \rp(t,\vx)$, $\delta \rp(T,\vx)$ are the values of the first variation functions of $\rp(s,\vx)$ at the fixed time boundaries namely $s = t$ and $s = T$. Here 
%	\begin{equation}
	$\delta \rp(t,\vx) = \delta \rp(s,\vx)\Big|_{s = t} - \frac{\partial \rp}{\partial s}(s,\vx) \Big|_{s = t} \delta (t)$
%	\end{equation}
	and
%	\begin{equation}
	$\delta \rp(T,\vx) = \delta \rp(s,\vx)\Big|_{s = T} - \frac{\partial \rp}{\partial s}(s,\vx) \Big|_{s = T} \delta (T)$, 
%	\end{equation}
	where $\delta \rp(s,\vx)|_{s}$ denotes the variation of $\rp(s,\vx)$ at time time $s$. We also know that the variation of $\rp$ at $s = t$ is zero or $\delta \rp|_{s = t} = 0$ since the distribution $\rp$ at the boundary $s = t$ or at the initial instant of time is known to be $\rp_0(\vx)$ and that $\delta T = 0$ since we assume a fixed time boundary. Therefore equations $\eqref{5.12}, \eqref{5.13}, \eqref{5.14}$ imply	
	\begin{align}
	&\delta J^\star \left(t; \rp, \vu \right) = \bigg<(\phi(s,\vx)\big|_{T} - \pi(s,\vx)\big|_T),\delta \rp(s,\vx)|_{T}\bigg> \nonumber \\
	&+ \int\limits_{t}^{T} \; \bigg[ \bigg<\ell(s,\vx,\vu) + \mathcal{F}^{\dagger \; \vu(s)}_{\text{\tiny MJD}} \pi(s,\vx) + \frac{\partial }{\partial s}\pi (s,\vx),\delta \rp(s,\vx)\bigg>  \bigg] \; \rd s \notag \\
	&+ \int\limits_{t}^{T} \; \calH^{\rT}_\vu(s,\vu(s);\rp(s,\vx),\pi(s,\vx)) \delta \vu(s) \; \rd s, \label{5.15}
	\end{align}
	\noindent because $\delta \rp(t,\vx) = 0$ as explained earlier in comments on equation \eqref{5.14}. The variations $\delta \rp(s,\vx)|_T$, $\delta \vu(s)$ and $\delta \rp(s,\vx)$ which appear in the above equation are arbitrary and are non zero, so that the three terms above are independent of each other. Therefore, by using the Fundamental Lemma of the Calculus of Variations, with the usual mild conditions \cite{Weinstock1952} and equation \eqref{5.15}, we have that $\delta J^\star \left(t; \rp, \vu \right) = 0$ implies equations \eqref{Euler_Lagrange_1} \eqref{Euler_Lagrange_2}, \eqref{Euler_Lagrange_3}.
%	\begin{equation}
%	\calH_{\vu}(s,\vu(s);\rp(s,\vx),\pi(s,\vx)) = 0 \label{Euler_Lagrange_1}
%	\end{equation}	
%	%\label{5.16}	
%	\begin{equation}
%	-\frac{\partial \pi(s,\vx)}{\partial t} = \ell(s,\vx, \vu(t)) + \mathcal{F}^{\dagger \vu(s)}_{\text{\tiny MJD}} \pi(s,\vx)  \label{Euler_Lagrange_2}
%	\end{equation}
%	%\label{5.17}	
%	\begin{equation}
%	\pi(T,\vx) = \phi(T,\vx). \label{Euler_Lagrange_3}
%	\end{equation}
	%\label{5.18}	
	\noindent These are the conditions for minimizing $J\big(t;\rp, \vu\big)$ using $\vu$ subject to the governing dynamics of the Kolmogorov Feller PDE. Explicit formulation of $\calH_\vu$ can be obtained due to partial differentiatiability conditions w.r.t. $\vu$ as given by \eqref{Hamlitonian_IDMP_Control_Gradient}.
%	\begin{align}
%	&\calH_\vu(s,\vu;\rp,\pi) = \bigg<\ell_\vu + \sum\limits_{i = 1}^{n_x}\frac{\partial F_i(s,\vx,\vu)}{\partial \vu} \frac{\partial \pi(s,\vx)}{\partial x_i} + \frac{1}{2}\sum\limits_{i,j = 1}^{n_x}\frac{\partial}{\partial \vu} \bigg( \frac{\partial^2  (\Sigma_{ij}(s,\vx,\vu)\pi(s,\vx))}{\partial x_i \partial x_j} \bigg),\rp(s,\vx) \bigg>, \label{Hamlitonian_IDMP_Control_Gradient}
%	\end{align}	
	\noindent recalling that $ \rd_{\text{jump}} \pi(\vx,t)$ does not have explicit dependence on $\vu$.	 
\end{proof}
 
%{We give the following definitions relating to the IDMP optimality system in theorem \ref{Thm3}}

\begin{definition}
	Consider the necessary conditions for the solution of the optimal control problem \eqref{Problem}. If there exists an admissible control $\vu^*(t)$ and a corresponding IDMP costate function $\pi^*(t,\vx)$, such that the Euler-Lagrange equations \eqref{Euler_Lagrange_1}, \eqref{Euler_Lagrange_2}, \eqref{Euler_Lagrange_3} are satisfied at time $t$, then they are called an infinite dimensional optimal control policy and the corresponding optimal IDMP costate function at time $t$. The PDF for $\rp^*(t,\vx)$ denotes the the corresponding optimal PDF satisfying the forward Chapman-Kolmogorov PIDE \eqref{Forward_CKPDE} under the optimal control.
\end{definition}

%%%%%%%%%%%%%%%%%%%%%%%%%%%%%%%%%%%%%%%%%%%%%%%%%%%%%%%%%%%%%%%%%%%%%%%%%%%%%%%%%%%%%%%%%%%%%%%%%%%%%%%%%%%%%%%%%%%%%%%%%%%%%%%%%%%%%%%%%%%%%%%%%%%%%%%%%%%%%%%%%%%%%%%%%%%%%%%%%%%%%%%%%%%%%%%%%%%%%%%%%%%%%%%%%%%%%%%%%%%%%%%%%%%%%%%%%%%%

\section{Relationship with Dynamic Programming Principle}
\label{sec:Relationship}

\subsection{Linear Feynman-Kac lemma and the SDP connection}

Under certain conditions, Dynkin's formula \cite{Hanson07appliedstochastic} formula for the backward costate PIDE \eqref{Euler_Lagrange_2} governing the optimal costate function gives
\begin{align}
\pi^*(t,\vx) &= \Eb \Big[\phi(T,\vx_T) + \int_{t}^{T} \ell(s,\vx_s,\vu^*(s)) \rd t |\vx_t = \vx \Big]. \label{FK}
\end{align}
\noindent Applying the iterated expectations property of conditional expectations  (details in Section \ref{sec:SamplingBasedAlgorithms})  we have
\begin{equation}
\pi^*(t,\vx) = \Eb \big[ \ell (t,\vx,\vu_t) \rd t + \pi^*(t + \rd t, \vx_{t + \rd t}) \big|\vx_t = \vx \big]. \label{ModifiedFeynmanKac}
%&\approx  \ell (t,\vx,\vu_t) \rd t + \Eb_{\rp( \vx^{'}|\vx)} \big[ \pi^*(t + \rd t,\vx + \rd \vx_t) \big|\vx_t = \vx \big] 
\end{equation}
\noindent This expression of the optimal costate indicates a relationship of this function with the SDP principle \cite{YongBook_1958}. We investigate this relationship in the next subsection. It is well known that SDP can be applied to the stochastic problem \eqref{Problem} only when the initial state is known with probability one. It is natural to then investigate, the relationship between a version of the DPP for non degenerate initial distributions and the IDMP optimality system.  This is the topic of subsection \ref{sec:Relationship.DPPandIDMP}. For brevity we abuse our notation by curtailing the expression of dependent variables whenever necessary in this section.

\subsection{Connection between IDMP and SDP}
\label{sec:Relationship.IDMPandSDP}

%Equation \eqref{ModifiedFeynmanKac} motivates an exploration of the connection of the IDMP optimality system with the HJB PIDE derived via SDP. The key idea is that application of SDP is restricted to the case of the initial state distribution being a degenerate delta Dirac distribution.

\subsubsection{HJB theory for control of Q-marked Jump Diffusion SDEs}
\label{subsec:HJB_PDE_MJDPs}

We recall the Hailton Jacobi Bellman (HJB) control theory for QMJD processes here. The optimality system stated will be referred to in the next subsection to see its similarity with a form of the costate PIDE. Consider the QMJD process in (\ref{QMJDP}). The value function is defined as the optimal cost to go at any point of time. We assume that the running and terminal cost functions are chosen such that they lead to a cost which is integrable at all instants of time. The value function may then be written as
\begin{align}
v(t,\vx) = \underset{\vu}{\min} \; \Eb\left[\phi(T,\vx_T) + \int\limits_{t}^{T} \ell(t,\vx_t,\vu_t) \; \rd t \Bigg|\vx_{t} = \vx\right]. \label{Value_function}
\end{align}
%Let us define the HJB Hamiltonian operator for the HJB PIDE as follows.
\begin{definition}
	\textit{We define the HJB Hamiltonian operator for the HJB PIDE corresponding to the QMJD process \eqref{QMJDP}, when it exists, by}
	\begin{align}
	\calH^{{\text{\tiny  HJB}}}\big(t, \vx, \vu, v(t,\vx)) \triangleq& \; \ell(t,\vx,\vu) + v_\vx^\rT(t,\vx) \vF(t,\vx,\vu) + \frac{1}{2} \rtr(\vSigma v_{\vx\vx})(t,\vx,\vu) + \rd_{\text{jump}} v(t,\vx) \label{Hamiltonian_HJB}
	\end{align}      	
	\noindent \textit{where $v:[0,T] \times \Rb^{n_x} \rightarrow \Rb$ is the value function and $\ell$ is the running cost function.}
\end{definition} 	
Theorem 6.3 (pp 177) in \cite{Hanson07appliedstochastic} states that if
%\begin{itemize}
	(A1) the decomposition rules (pp 172, Rules 6.1) \cite{Hanson07appliedstochastic} hold,
	(A2) there exists a value function $v$ such that $v\in C^{1,2}([0,T]\times \Rb^{n_x})$,
	(A3) there exists an optimal control, called the HJB optimal control, given by
%\end{itemize} 
\begin{equation}
\vu_{\text{\tiny HJB}}^*(t,\vx) = \underset{\vu \in D_\vu}{\argmin} \; \calH^{\text{\tiny HJB}}(t,\vx,\vu, v) \; \text{ for all } (t,\vx) \in [0,T]\times \Rb^{n_x},
\end{equation}
then $v$ satisfies HJB equation for QMJD processes for all $(t,\vx) \in [0,T]\times \Rb^{n_x}$
\begin{align}
%-v_t(t,\vx) =& \ell(t,\vx,\vu_{\tiny \text{HJB}}^*) + v_\vx^\rT(t,\vx) \vF(t,\vx,\vu_{\tiny \text{HJB}}^*) \notag \\
%&+ \frac{1}{2} \rtr(\vSigma v_{\vx \vx})(t,\vx,\vu_{\text{HJB}}^*) + \rd_{\text{jump}} v (t,\vx)  \label{HJB_QMJ} \\
-v_t(t,\vx) =& \underset{\vu \in D_\vu}{\min} \; \calH^{\text{\tiny \text{HJB}}}(t,\vx,\vu, v(t,\vx))  \label{HJB_QMJD} \\
v(T,\vx) =& \phi(T,\vx). \label{HJB_QMJDterminal}
\end{align}
%\noindent Under the additional assumptions \\
%\noindent (C1) $\ell(t,\vx,\vu) = q(t,\vx) + \frac{1}{2} \vu^\rT R \vu, \; R \in \Rb^{n_u \times n_u} \text{ and } R > 0$ \\
%\noindent (C2) $\vF(t,\vx,\vu) = \vf(t,\vx) + \vG(t,\vx) \vu$ \\
%\noindent for the value function \eqref{Value_function} on the MJDP \eqref{MJD} we would have (section III) \cite{EvangelosACC2010} for the control affine quadratic cost case, the specific HJB optimal control given by
%\begin{equation}
%\vu_{\text{\tiny HJB}}^*(t,\vx) = -R^{-1}\vG^\rT (t,\vx) v_\vx(t,\vx), \; \forall(t,\vx) \in [0,T] \times \Rb^{n_x} \label{HJB_Optimal_control}
%\end{equation}
%\noindent and the HJB PDE which the value function satisfies given by
%\begin{align}
%-v_t(t,\vx) =& q(t,\vx) + (v_\vx^\rT \vf)(t,\vx) - \frac{1}{2} (v_\vx^\rT \vG R^{-1} \vG^\rT v_\vx)(t,\vx) + \frac{1}{2} \rtr(\vSigma v_{\vx\vx})(t,\vx) + \rd_{\text{jump}} v(t,\vx) \label{HJB_MJDP_Control_affine} \\
%v(T,\vx) =& \phi(T,\vx). \label{HJB_MJDP_Control_Affine_Boundary_Condition}
%\end{align}

\subsubsection{Generalized Optimal Costate equation and Relationship with HJB PDE}

{Under the optimal control, Euler-Lagrange equations \eqref{Euler_Lagrange_1}, \eqref{Euler_Lagrange_2}, can be stated in a concise form \eqref{InfDimHJB_Equation}, \eqref{InfDimHJB_Equationterminal}. This result provides the time rate of change of the optimal costate function integrated over the optimal PDF trajectory, assuming that the IDMP optimality conditions are satisfied. We call equation \eqref{InfDimHJB_Equation} with terminal condition \eqref{InfDimHJB_Equationterminal} as the generalized optimal costate equation. A remark at the end of this section describes the similarities and differences in the governing optimality systems obtained using IDMP and SDP under an additional assumption on problem \eqref{Problem}.} 

\begin{lemma} \label{Lemma1}
	Let $\vu^* \in \calV[t,T]$, $\pi^*, \rp^* \in C_c^{1,2}([t,T] \times \Rb^{n_x})$ be the infinite dimensional optimal control, optimal costate function and the corresponding optimal PDF for the problem \eqref{Problem}. If
			\begin{itemize}
				\item[(L1)] there exist unique $\vu^* \in \calV[t,T]$, $\pi^* \in C_c^{1,2}([t,T] \times \Rb^{n_x})$ satisfying the Euler-Lagrange equations \eqref{Euler_Lagrange_1}, \eqref{Euler_Lagrange_2}, \eqref{Euler_Lagrange_3} for all $s \in [t,T]\times\Rb^{n_x}$ and $\rp^* \in C_c^{1,2}([t,T] \times \Rb^{n_x})$ satisfies \eqref{Forward_CKPDE} under $\vu^*:[t,T]$
				\item[(L2)] $\calH(s,\vu;\rp,\pi^*)$ is convex and continuously differentiable w.r.t. $\vu$ on $[t,T] \times D_\vu$ 
%				\item[(L3)] $\big<\frac{\partial \pi^*}{\partial s}(s),\rp^*(s) \big>$ is bounded on $[t,T]$
			\end{itemize}
	then for all $s \in [t,T]$
	\begin{equation}
		-\left<\frac{\partial \pi^*}{\partial s}(s),\rp^*(s) \right> = \underset{\vu \in D_\vu}{\min}\calH(s,\vu;\rp^*,\pi^*) = \underset{\vu \in D_\vu}{\min} \big< \ell(s,\vu) + \mathcal{F}^{\dagger \; \vu}_\text{\tiny MJD} \; \pi^*(s), \rp^*(s) \big>. \label{InfDimHJB_Equation}
	\end{equation}
	\begin{equation}
		\text{and} \; \big<\pi^*(T),\rp^*(T) \big> = \big<\phi(T),\rp^*(T) \big>. \label{InfDimHJB_Equationterminal}
	\end{equation}
\end{lemma}
\begin{proof}
	Let $\vu^* \in D_\vu$ such that $\calH_\vu(s,\vu^*;\rp,\pi^*) = \mathbf{0}$ where $s \in [0,T]$. Conditions (L1) and (L2) then imply that
	\begin{equation}
		\vu^*(s) = \underset{\vu \in D_\vu}{\min} \calH(s,\vu;\rp^*,\pi^*),
	\end{equation}
	so that we may write
	\begin{align}
		&\calH(s,\vu^*;\rp^*,\pi^*) = \underset{\vu \in D_\vu}{\min} \calH(s,\vu;\rp^*,\pi^*) \notag \\
		&= \underset{\vu \in D_\vu}{\min} \big< \ell(s,\vu) + \mathcal{F}^{\dagger \; \vu}_\text{\tiny MJD} \; \pi^*(s),\rp(s) \big> = \big< \ell(s,\vu^*) + \mathcal{F}^{\dagger \; \vu^*}_\text{\tiny MJD} \; \pi^*(s),\rp^*(s) \big>. \label{OptimalHamiltonian}
	\end{align}
	Therefore the Euler-Lagrange equations \eqref{Euler_Lagrange_2} and \eqref{OptimalHamiltonian} give us
	\begin{align}
		&-\left<\frac{\partial \pi^*}{\partial s}(s),\rp^*(s) \right> = \big< \ell(s,\vu^*) + \mathcal{F}^{\dagger \; \vu^*}_\text{\tiny MJD} \; \pi^*(s),\rp^*(s) \big> \notag \\
		&= \underset{\vu \in D_\vu}{\min} \calH(s,\vu;\rp^*,\pi^*) = \underset{\vu \in D_\vu}{\min} \big< \ell(s,\vu) + \mathcal{F}^{\dagger \; \vu}_\text{\tiny MJD} \; \pi^*(s), \rp^*(s) \big>. 
	\end{align}
	The terminal condition \eqref{InfDimHJB_Equationterminal} is true because the optimal costate satisfies the terminal condition \eqref{Euler_Lagrange_3}.
\end{proof}

\textit{\textbf{Remark:}} \textit{It is well known that application of SDP for the SOC problem \eqref{Problem}, is restricted to the case when the initial distribution is specified $\rp(t,\vx) = \delta(\vx - \vy)$ where $\vy \in \Rb^{n_x}$. Applying this condition for the generalized optimal costate equation would derive
\begin{equation}
- \frac{\partial \pi^*}{\partial s}(t, \vy) = \underset{\vu \in D_\vu}{\min} \calH(t,\vu;\delta(\vx - \vy),\pi^*(t,\vx)).
\end{equation}
From the definitions of the HJB Hamiltonian operator \eqref{Hamiltonian_HJB} and Hamiltonian functional \eqref{Hamiltonian_IDMP}, it is observed that if $\rp(t,\vx) = \delta(\vx - \vy)$ then they are related for all $\vu, v$ at the initial time instant by
\begin{equation}
\calH(t,\vu;\delta(t,\vx - \vy),v) = \calH^{\text{\tiny HJB}}(t,\vy,\vu, v(t,\vy)). \label{HamiltoniansEqual}
\end{equation}
Equations \eqref{HamiltoniansEqual}, \eqref{HJB_QMJD} then imply
\begin{equation}
- \frac{\partial v}{\partial s}(t, \vy) = \underset{\vu \in D_\vu}{\min} \calH(t,\vu;\rp(t,\vx),v). \label{HJB_QMJD_deltaInit}
\end{equation}
We can therefore see that the optimal costate PIDE and HJB PIDE are idential at the initial time instant given $\rp(t,\vx) = \delta(\vx - \vy)$. However, in general, the equations satisfied by the optimal costate \eqref{InfDimHJB_Equation} and value function \eqref{HJB_QMJD} are distinct at all other time instants since Equation  \eqref{HamiltoniansEqual} is true only at the initial time instant. Observe, additionally, that the terminal conditions  \eqref{InfDimHJB_Equationterminal}, \eqref{HJB_QMJDterminal} are not equal because, $\rp(T)$ is in general, not degenerate. The optimal control generated by IDMP, $\vu^*(s) = {\argmin}_{\vu \in D_\vu} \calH(s,\vu;\rp^*,\pi^*)$, is therefore distinct from the one generated by SDP, $\vu_{\text{\tiny HJB}}^*(s,\vy) = {\argmin}_{\vu \in D_\vu} \; \calH^{\text{\tiny pseudo}}(s,\vy,\vu,v(t,\vy))$, at all time instants including the initial instant. In addition, we recall that the IDMP control is an open loop control which depends only implicitly on the PDF at any time instant, while the SDP control is explicitly a closed loop control.}

\subsection{DPP for the Infinite Dimensional Value function and Relationship with IDMP}
\label{sec:Relationship.DPPandIDMP}

Based on Lemma \ref{Lemma1} we now quantitatively establish the relationship between the IDMP and a DPP satisfied by the infinite dimensional value function. First we construct the DPP satisfied by such an infinite dimensional version of the value function for PIDE control of QMJD processes. Then we show how the infinite dimensional value function is related to the optimal costate function along the optimal PDF trajectory. {A precise expression of this relationship has not been stated clear way in preexisting works, as far as known to the authors.} Assuming sufficient smoothness of the costate function and PDF and that IDMP optimality conditions are satisfied, the relationship between IDMP and DPP is stated explicitly in what follows.
\begin{definition} 
	\textit{We define the {infinite dimensional} value function at $t \in [0,T)$ for the problem \eqref{Problem} if it exists and is finite, as}
	\begin{equation} 
		V(t;\rp(t)) = \underset{\vu \in \calV[t,T]}{\min} \; J(t;\rp,\vu) \label{Def:InfDimValueFunction}
	\end{equation}
	\begin{equation}
		\text{and} \; V(T;\rp(T)) = \big<\phi(T), \rp(T) \big>. \label{Def:InfDimValueFunctionTerminal}
	\end{equation}
\end{definition}

\begin{theorem} \label{Thm4}
	Let there exist a unique finite valued value function defined in \eqref{Def:InfDimValueFunction} for the problem \eqref{Problem} under the dynamics \eqref{5.3}. Then for all $s \geq t$
	\begin{equation}
		V(t;\rp(t)) = \underset{\vu \in \calV[t,T]}{\min} \bigg\{ \int\limits_{t}^{s}\big< \ell(\tau,\vu), \rp(\tau) \big> \rd \tau + V(s; \rp(s)) \bigg\} \label{IDBellman}.
	\end{equation}
\end{theorem}
\begin{proof}
	We have for all $\epsilon > 0$ there exists $\vu_\epsilon \in \calV[t,T]$ such that
	\begin{align}
		V(t;\rp(t)) + \epsilon \geq J(t;\rp,\vu_\epsilon) = \int\limits_{t}^{s}\big< \ell(\tau,\vu_\epsilon), \rp(\tau) \big> \rd \tau + J(s;\rp,\vu_\epsilon) \geq \int\limits_{t}^{s}\big< \ell(\tau,\vu_\epsilon), \rp(\tau) \big> \rd \tau + V(s;\rp) 	
	\end{align}
	so that the following minimum on the right hand side satisfies the inequality
	\begin{equation}
		 V(t;\rp(t)) + \epsilon \geq \underset{\vu \in \calV[t,T]}{\min} \bigg \{ \int\limits_{t}^{s}\big< \ell(\tau,\vu), \rp(\tau)\big> \rd \tau + V(s;\rp) \bigg \}. \label{1}
	\end{equation}
	Now given $\epsilon > 0$ and $\vu_\epsilon \in \calV[t,T]$, we can choose $\vu(\tau) = \vu_\epsilon(\tau)$ when $\tau \in [t,s]$ such that $V(s;\rp(s)) + \epsilon \geq J(s;\rp,\vu_\epsilon)$. From Equation \eqref{Def:InfDimValueFunction} of Definition 6 it can be seen that for all $\vu_\epsilon \in \calV[t,T]$, $\epsilon > 0$,
	\begin{align}
		V(t;\rp(t)) + \epsilon \leq J(t;\rp,\vu) = \int\limits_{t}^{s}\big< \ell(\tau,\vu), \rp(\tau) \big> \rd \tau + J(s;\vu,\rp) \leq \int\limits_{t}^{s}\big< \ell(\tau,\vu_\epsilon), \rp(\tau) \big> \rd \tau + V(s;\rp(s)) + \epsilon. \label{2}
	\end{align}
%Thus we will have that 
% \begin{align}
%	V(t;\rp(t))  - \epsilon \leq& \int\limits_{t}^{s}\big< \ell(\tau,\vu_\epsilon), \rp(\tau) \big> \rd \tau + V(s;\rp(s))  \label{3}
%	\end{align}
	Since this inequality is true for all $\vu_\epsilon \in \calV[t,T]$, we may take the minimum over all such $\vu_\epsilon$ 
	\begin{align}
		V(t;\rp(t)) \leq& J(t;\vu,\rp) \leq \underset{\vu_\epsilon \in \calV[t,T]}{\min} \bigg \{ \int\limits_{t}^{s}\big< \ell(\tau,\vu_\epsilon), \rp(\tau) \big> \rd \tau + V(s;\rp) \bigg \} \notag \\
		\leq& \underset{\vu \in \calV[t,T]}{\min} \bigg \{ \int\limits_{t}^{s}\big< \ell(\tau,\vu), \rp(\tau) \big> \rd \tau + V(s;\rp) \bigg \} + \epsilon. \label{3}
	\end{align}
	where we have replaced the symbol $\vu_\epsilon$ with $\vu$ in the last equality. Combining equations \eqref{1}, \eqref{3} implies
	\begin{equation}
		V(t;\rp(t)) - \epsilon \leq \underset{\vu \in \calV[t,T]}{\min} \bigg \{ \int\limits_{t}^{s}\big< \ell(\tau,\vu), \rp(\tau) \big> \rd \tau + V(s;\rp) \bigg \} \leq V(t;\rp(t)) + \epsilon.
	\end{equation}
	for all $\epsilon > 0$. Under the limit $\epsilon \rightarrow 0$ the desired result is obtained.
\end{proof}

The following theorem quantitatively states the relationship between the infinite dimensional value function and the optimal costate function under the assumption that problem \eqref{Problem} has a unique solution.

\begin{theorem} \label{Thm5}
		Let $\vu^* \in \calV[t,T]$, $\pi^*, \rp^* \in C_c^{1,2}([t,T] \times \Rb^{n_x})$ be the unique infinite dimensional optimal control, optimal costate function and corresponding PDF for the problem \eqref{Problem}. If (L1) and (L2) are true and
		\begin{itemize}
			\item [(T5)] there exists a unique finite valued value function defined in \eqref{Def:InfDimValueFunction} for the problem \eqref{Problem}
		\end{itemize}
		then for all $s \in [t,T]$
		\begin{equation}
			V(s;\rp^*(s)) = \big< \pi^*(s), \rp^*(s) \big>.
		\end{equation}
\end{theorem}
\begin{proof}
Dynkin's formula for the jump diffusion process \eqref{QMJDP} can be applied to $\pi^*(s,\vx) \in C_c^{1,2}([t,T] \times \Rb^{n_x})$ which satisfies \eqref{Euler_Lagrange_2}, \eqref{Euler_Lagrange_3} by following Theorem 7.1, chapter 7 in \cite{Hanson07appliedstochastic} to obtain
\begin{equation}
	\pi^*(s,\vx) = \Eb \left[ \int\limits_{s}^{T} \ell(\tau,\vx_\tau,\vu^*(\tau)) \rd \tau + \phi(T,\vx_T) \big| \vx_s = \vx \right]
\end{equation}
where the expectations are under the optimal PDF governed by \eqref{5.3} under the control $\vu^*(\tau) \in \calV[s,T]$ is $\rp^*(\tau):[s,T]$. The law of iterated expectations implies for all $t \leq s \leq \hat{s} \leq T$
\begin{align}
	\pi^*(s,\vx) 
%			=& \Eb \bigg[ \; \Eb \bigg[ \int\limits_{t}^{T} \ell(s,\vx_s,\vu^*(s)) \rd s + \phi(T,\vx_T) \bigg| \vx_{t + \rd t}] \; \bigg| \vx_t = \vx \bigg] \notag \\
	=& \; \Eb \bigg[ \; \Eb \bigg[ \int\limits_{s}^{\hat{s}} \ell(\tau,\vx_\tau,\vu^*(\tau)) \rd \tau + \int\limits_{\hat{s}}^T \ell(\tau,\vx_\tau,\vu^*(\tau)) \rd \tau + \phi(T,\vx_T) \bigg| \vx_{\hat{s}} \bigg] \; \bigg| \vx_s = \vx \bigg] \notag \\
	=& \; \Eb \bigg[ \int\limits_{s}^{\hat{s}} \ell(\tau,\vx_\tau,\vu^*(\tau)) \rd \tau + \Eb \bigg[ \int\limits_{\hat{s}}^T \ell(\tau,\vx_\tau,\vu^*(\tau)) \rd \tau + \phi(T,\vx_T) \bigg| \vx_{\hat{s}} \bigg] \; \bigg| \vx_s = \vx \bigg] \notag \\
	=& \; \Eb \bigg[ \int\limits_{s}^{\hat{s}}\ell(\tau,\vx_\tau,\vu^*(\tau)) \rd \tau + \pi^*(\hat{s}, \vx_{\hat{s}}) \bigg| \vx_s = \vx \bigg].
\end{align}
Using the law of total expectation we have
\begin{align}
	&\big< \pi^*(s), \rp^*(s) \big> = \Eb \left[ \; \int\limits_{s}^{\hat{s}} \Eb \big[ \ell(\tau,\vx_\tau,\vu^*(\tau)) \rd \tau + \pi^*(\hat{s}, \vx_{\hat{s}}) \big| \vx_s = \vx \big] \right] \notag \\
	&= \Eb \bigg[\int\limits_{s}^{\hat{s}} \ell(\tau,\vx_\tau,\vu^*(\tau)) \rd \tau \bigg] + \Eb \bigg[ \pi^*(\hat{s}, \vx_{\hat{s}}) \bigg] = \int\limits_{s}^{\hat{s}} \big< \ell(\tau, \vu^*(\tau)), \rp({\tau})) \big> \rd \tau + \big< \pi^*(\hat{s}), \rp({\hat{s}}) \big>. \label{4}
\end{align}
Due to the fact that the optimal control $\vu^* \in \calV[t,T]$ solving the problem \eqref{Problem} is unique and (T5), we can write the following result from Theorem \ref{Thm4}. Note again that we denote the optimal PDF $\rp^*(\tau):[s,T]$ evolving under the optimal control $\vu^* \in \calV[s,T]$.
\begin{align}
	V(s;\rp(s)) =& \underset{\vu \in \calV[t,T]}{\min} \left\{ \int\limits_{s}^{\hat{s}}\big< \ell(\tau,\vu), \rp(\tau) \big> \rd \tau + V(\hat{s}; \rp(\hat{s})) \right\} = \int\limits_{s}^{\hat{s}}\big< \ell(\tau,\vu^*), \rp^*(\tau) \big> \rd \tau + V(\hat{s}; \rp^*(\hat{s})) \label{5}
\end{align}
Recall the terminal conditions for the value function \eqref{Def:InfDimValueFunctionTerminal} and for optimal costate function \eqref{Euler_Lagrange_3} implies for all $\rp(T)$ satisfying the conditions (T1) through (T4)
\begin{equation}
V(T;\rp(T)) = \big< \pi^*(T), \rp(T) \big> = \big< \phi(T), \rp(T) \big> \label{6}
\end{equation}
which is true for $\rp^*(T)$ as well. Let us choose $\hat{s} = T$ in equation \eqref{3}, \eqref{4}. Observing that equations \eqref{4}, \eqref{5} are identical and the value function is unique due to (T5), we can prove easily using equation \eqref{6} that $V(s;\rp^*(s)) = \big< \pi^*(s), \rp^*(s) \big>$ along the optimal PDF trajectory $\rp^*(s):[s,T]$ under the control $\vu^* \in \calV[s,T]$.
\end{proof}
Stated in words, we have proved that the infinite dimensional value function is equal to the $\mathcal{L}^2$ product of optimal costate function with the optimal PDF along the optimal trajectory. Note that this relationship was proved using the mechanism of the linear Feynman-Kac lemma, which motivated our investigation.

\section{Sampling based algorithm for PDF Control of QMJD processes}
\label{sec:SamplingBasedAlgorithms}

Using the Feynman-Kac formula \eqref{FK} directly, to compute the costate or optimal costate by forward sampling would be computationally prohibitive. Direct application would require generating samples over the entire time horizon starting from each space time grid point. Instead, we use an iteratively backpropagated costate (IBC) algorithm \cite{Bakshi_CDC2016}, the key ingredient for which is derived below. By Dynkin's formula, Theorem 7.1, chapter 7 of \cite{Hanson07appliedstochastic} for QMJD process \eqref{QMJDP}, applied to $\pi(t,\vx) \in C_c^{1,2}([0,T] \times \Rb^{n_x})$ which satisfies \eqref{Euler_Lagrange_2}, \eqref{Euler_Lagrange_3} under arbitrary control $\vu(s) \in \calV[t,T]$
\begin{equation}
	\pi(t,\vx) = \Eb \left[ \int\limits_{t}^{T} \ell(s,\vx_s,\vu(s)) \rd s + \phi(T,\vx_T) \big| \vx_t = \vx \right].
\end{equation}
\noindent The law of iterated expectations implies
\begin{align}
	\pi(t,\vx) =& \; \Eb \bigg[ \; \Eb \bigg[ \ell(t,\vx_t,\vu(t)) \rd t + \int\limits_{t + \rd t}^{T} \ell(s,\vx_s,\vu(s)) \rd s + \phi(T,\vx_T) \bigg| \vx_{t + \rd t} \bigg] \; \bigg| \vx_t = \vx \bigg] \notag \\
	=& \; \Eb \bigg[ \ell(t,\vx_t,\vu(t)) \bigg]\rd t + \Eb \bigg[ \int\limits_{t + \rd t}^{T} \ell(s,\vx_s,\vu(s)) \rd s + \phi(T,\vx_T) \bigg| \vx_{t + \rd t} \bigg] \; \bigg| \vx_t = \vx \bigg] \notag \\
	=& \; \Eb \bigg[ \ell(t,\vx_t,\vu(t)) \rd t + \pi^*(t + \rd t, \vx_{t + \rd t}) \bigg| \vx_t = \vx \bigg].
\end{align}

We denote the temporal grid indexed as $[t_0,t_N] = [0,T]$. We have dropped the conditional expectation notation for brevity in the following pseudo code and pick a small number $\epsilon > 0$.
%\vspace{-0.5cm}
%\begin{algorithm} 
%	\caption{PDF control of Q-MJD processes} \label{Algorithm:1}
%	\begin{algorithmic}[1]
%		\State \textbf{Initialize} Choose $\vu^0_t:[t_0,t_N]$ arbitrarily. 
%		\Repeat
%		\State Compute $\pi^k(t,\vx)\text{ on }(t,\vx)$ by \eqref{5.2.3}.
%		\State Compute $\calH^k_{\vu}(t;\rp^k,\vu^k_t,\pi^k)$ on $[t_0,t_N]$ by \eqref{5.19}.
%		\State Update control: $\vu^{k+1}_t = \vu^k_t - \epsilon \calH_{\vu}^k(t)$.
%		\Until{Convergence $\calH_{\vu} = 0$.} \\
%		\Return $\vu^*_t:[0,T]$.
%	\end{algorithmic}
%\end{algorithm}
%\vspace{-0.3cm}
%  \lipsum[1]

%\lipsum
%\begin{algorithm}
%	\caption{IBC PDF control of MJD processes} \label{Algorithm:2}
%	\begin{algorithmic}[1]
%		\State \textbf{Initialize} Choose $\vu^0_t:[t_0,t_N]$ arbitrarily. 
%		\Repeat
%		\State \textbf{Initialize} $\pi^k(t_N,\vx) = \phi(t_N,\vx)$.
%		\While {$i \neq 0$}
%		\State {$\pi^k(t_{i-1},\vx) = \ell(t_{i-1},\vx,\vu^k_{t_{i-1}}) +  \Eb\big[\pi^k(t_{i},\vx + \rd \vx_{t_i})]$}.
%		\State $i = i - 1$.
%		\EndWhile
%		\State Compute $\calH^k_{\vu}(t;\rp^k,\vu^k_t,\mu_\vQ^k)$ on $[t_0,t_N]$ by \eqref{Hamlitonian_IDMP_Control_Gradient}.
%		\State Update control: $\vu^{k+1}_t = \vu^k_t - \epsilon \calH_{\vu}^k(t)$.
%		\Until{Convergence $|\calH_{\vu}(t)| < \epsilon$.} \\
%		\Return $\vu^*_t:[0,T]$.
%	\end{algorithmic}
%\end{algorithm}
%
%
%
%
%\subsection*{Example Problem}

\label{sec:Examples}
\begin{wrapfigure}{L}{0.5\textwidth}
	\begin{minipage}{0.5\textwidth}
		\begin{algorithm}[H]
			\caption{assignment algorithm}
			\caption{IBC PDF control of MJD processes} \label{Algorithm:2}
			\begin{algorithmic}[1]
				\State \textbf{Initialize} Choose $\vu^0_t:[t_0,t_N]$ arbitrarily.
				\Repeat
				\State \textbf{Initialize} $\pi^k(t_N,\vx) = \phi(t_N,\vx)$.
				\While {$i \neq 0$}
				\State {$\pi^k(t_{i-1},\vx) = \ell(t_{i-1},\vx,\vu^k_{t_{i-1}}) +  \Eb\big[\pi^k(t_{i},\vx + \rd \vx_{t_i})]$}.
				\State $i = i - 1$.
				\EndWhile
				\State Compute $\calH^k_{\vu}(t;\rp^k,\vu^k_t,\mu_\vQ^k)$ on $[t_0,t_N]$ by \eqref{Hamlitonian_IDMP_Control_Gradient}.
				\State Update control: $\vu^{k+1}_t = \vu^k_t - \epsilon \calH_{\vu}^k(t)$.
				\Until{Convergence $|\calH_{\vu}(t)| < \epsilon$.} \\
				\Return $\vu^*_t:[0,T]$.
			\end{algorithmic}
		\end{algorithm}
	\end{minipage}
\end{wrapfigure}

%\subsubsection*{Problem set up}

\noindent  \textbf{Example Problem:} We demonstrate our algorithm for open loop control of ensembles with dynamics \eqref{QMJDP} with linear drift term, nonlinear diffusion coefficient and constant jump rate parameter \small
\begin{align} 
\rd x_t =& (- \alpha x_t + u(t)\rd t + \zeta \sqrt{\frac{(\kappa - x_t)^2}{2} + u(t)^2} \rd w_t \notag \\
&+ h(x_t,Q) \rd P_t, \label{7.1}
\end{align}
\normalsize
\noindent where $h(x_t,Q) = 0.5\cdot Q\cdot x_t$ with constant jump rate $\lambda = 1$ and mark density of $\textit{unif}([0,1])$. The initial condition is assumed to be a normal distribution. The state space is specified by the constraints $x(t) \in [-3,3]$ while the control is constrained by $u(t) \in [-3,3]$. Further two obstructions are modeled by the state constraints $x(t) \in [-2,-1] \; \text{at} \; t=1$ and $x(t) \in [0.5,2]  \; \text{at} \; t=2$. The process is defined to terminate on reaching any of the above boundaries. The values of the constants used are $\kappa = 3$, $\alpha = 0.5$, terminal time $T = 3$ and $\zeta = 0.1$. The task is to reach the target $x_{goal}(T) = 0$.

In this example we choose the running and terminal cost functions as $\ell(u) = \frac{R}{2}u^2$, $\phi(x) = Q_f (x - x_{goal})^2$, and a trajectory termination penalty of $\Xi - \tau$ where $\tau \in [0,T]$ is the stopping time of a trajectory colliding with a boundary or obstacle. Let $R = 2\times 10^{-4}$, $Q_f = \frac{4}{9}$, $\zeta = 0.1$ and $\Xi = 7$. The temporal discretization $\{t_i\}_{0 \leq i \leq N}$ is chosen to satisfy $\lambda \cdot (t_i - t_{i-1}) << 1$ the zero one law \cite{Hanson07appliedstochastic} allowing our use of the derived form of the PIDEs. Since we generate an open loop policy, we adopt the strategy of generating an implicity feedback policy. We do this by assuming $u(t) = u_1(t) + xu_2(t)$ and treating $\vu(t) = [u_1(t) \; u_2(t)]^\rT$ as the control we compute. This state parameterized policy results in a lower state dependent cost seen in Subfigure (1d). Trajectories are sampled in two steps in our algorithm. We sample single time step trajectories inside the costate computation loops at each spatio temporal grid point. We need full time horizon samples to compute the control gradient of the Hamiltonian at each control update iteration. Let $x_k(t)$ be the $k^{th}$ sample of trajectories at time $t$ for either case and $\mathbf{1}_k = \mathbf{1}_{\{\Xi_k < T\}}$ which indicates whether the sample was terminated by collision. Theorem \ref{Thm3}, Eq. \eqref{ModifiedFeynmanKac} imply
%\begin{align} \label{7.6}
%\calH(s,\vu(s);\rp,\pi) &= \bigg< \frac{R}{2}(u_1^2(t) + u_2^2(t)) + (- \alpha x + u(t)) \pi_x + \frac{1}{2}\zeta^2 \left(\frac{(3 - x)^2}{2} + (u_1^2(t) + u_2^2(t))\right) \pi_{xx} \nonumber \\
%&\quad \; + (\pi(x + h(x(t);Q),t) - \pi(x,t)), p(t,x) \bigg>
%\end{align}
\begin{align} \label{7.8}
&\pi(x,t) =  \frac{1}{K} \sum\limits_{k = 1}^{K} \left[ \phi(x_k(T)) + \int\limits_{t}^{T} \ell(x_k(s),u_s)\rd s \right](1 - \mathbf{1}_k) + \frac{1}{K} \sum\limits_{k = 1}^{K} \left[ (\Xi_k - \tau) + \int\limits_{t}^{\tau} \ell(x_k(s),u_s) \rd s \right]\mathbf{1}_k, \notag
\end{align}
\vspace{-3mm}
%\noindent with the Hamiltonian gradient being
\begin{equation}
\calH_\vu(s,\vu(s);\rp,\pi) = \begin{bmatrix}
R u_1(t) + \pi_x + \xi^2 u_1(t) \pi_{xx} \\ R u_2(t) + \pi_x + \xi^2 u_2(t) \pi_{xx}
\end{bmatrix}.
\end{equation}
%\noindent Here 100 one-time step samples and 1000 samples are used to compute the one step costate function update and the control gradient of the Hamiltonian respectively.
%We approximate the costate function at each time instant as a piecewise linear or cubic polynomial for interpolation.

%\subsubsection*{Results}

\noindent \textbf{Results:} Optimal costate function, a set of optimally controlled trajectories and the cost per iteration depicting convergence for $\rp_Q = \textit{unif}(0,1)$ and $\rp_0 = \calN(-1,\frac{1}{2})$ are illustrated in Subfigures (1a), (1b, (1c). Converged costs are compared in Subfigure (1c) with different initial conditions. The cost is lower for initial condition $\calN(-1,\frac{1}{2})$ since far lower number of optimally controlled trajectories end up colliding with the first obastacle depicted in the optimal trajectory sample in Subfigure (1b). We compare converged costs for simple jump diffusion $Q \sim \delta(1)$ with initial condition $\rp_0 = \delta(0)$, when using the state parameterized policy and non parameterized control in Subfigure (1d). This shows the benefit of the implicit feedback provided by the state parameterized policy.

\begin{figure}[ht]
	\centering
	\subfigure[]{%
		\includegraphics[width =  0.23\textwidth]{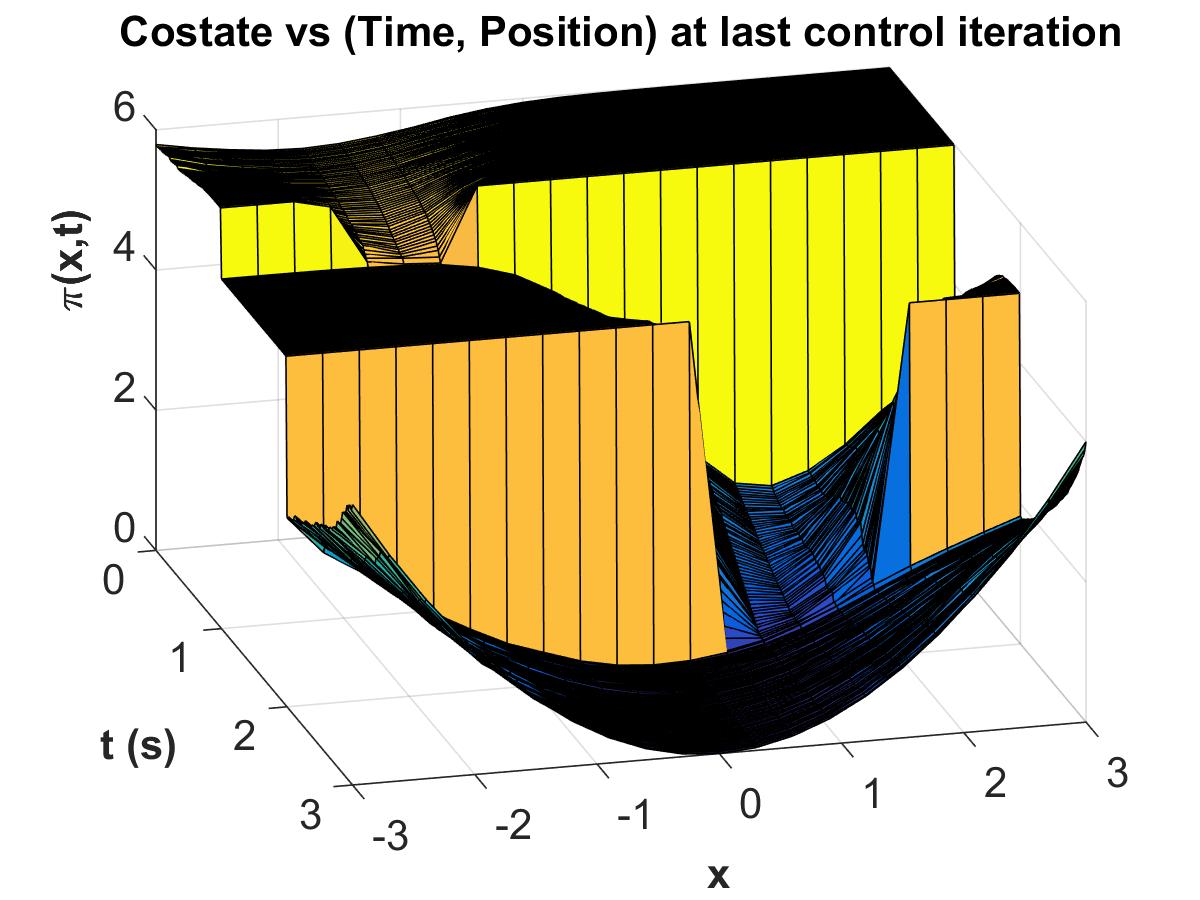}}
	\label{fig:Optimal_costate}
	\subfigure[]{%
		\includegraphics[width =  0.23\textwidth]{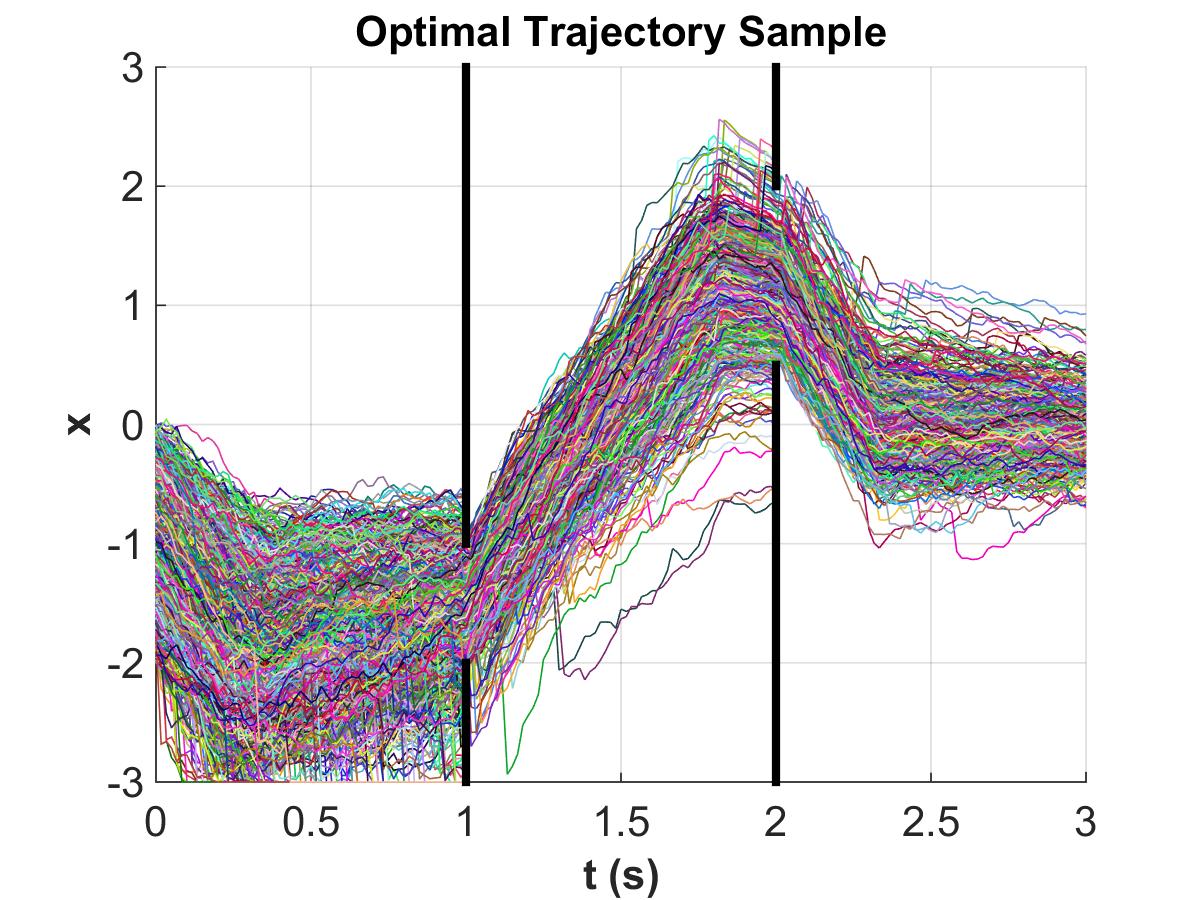}}
%	\caption{(a) Costate at last iteration $\rp_0 = \calN(-1,\frac{1}{2})$  (b)  Optimal trajectory samples for $\rp_0 = \delta$}
	\label{fig:Cost}
	\subfigure[]{%
		\includegraphics[width =  0.23\textwidth]{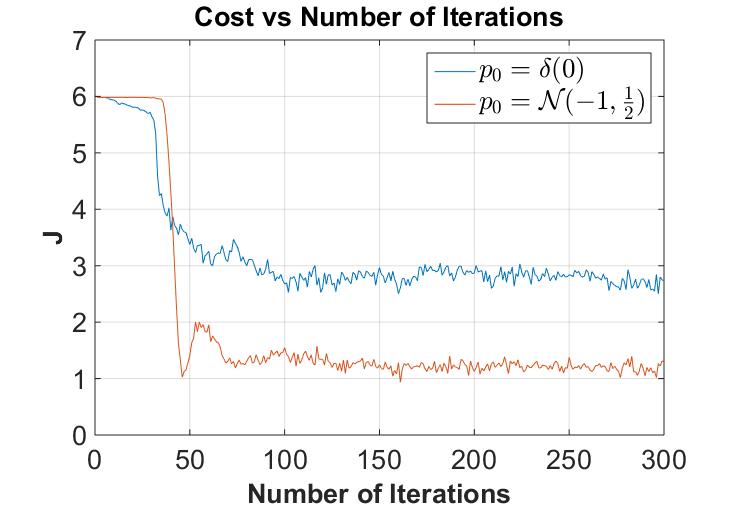}}
	\label{fig:Optimal_Samples}
	\subfigure[]{%
		\includegraphics[width =  0.23\textwidth]{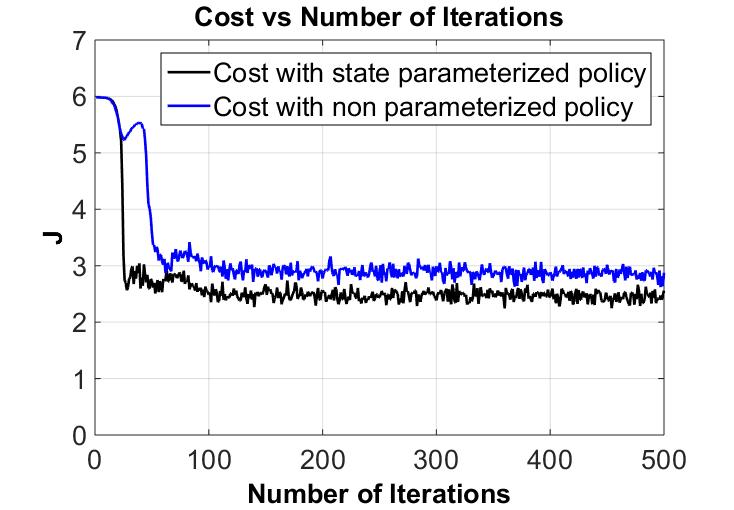}}
	\label{fig:Cost_comparison}
	\caption{(a) Costate at last iteration $\rp_0 = \calN(-1,\frac{1}{2})$  (b)  Optimal trajectory samples for $\rp_0 = \calN(-1,\frac{1}{2})$ (c) Cost vs iterations for $\rp_0 = \calN(-1,\frac{1}{2})$ in blue, and $\rp_0 = \delta(0)$ in red, (d)  Cost comparison with state parameterized policy in black, for $\rp_0 = \delta(0)$, $Q \sim \delta(1)$}
\end{figure}

\section{Conclusions}
\label{sec:Conclusions}

The open-loop deterministic formulation of the OCP results in a forward sampling-based algorithm to calculate the control by evaluating the co-state on the state space. This algorithm has the advantage of being \textit{na\"ively parallelizable} (applicable to density control of Q-MJD processes), an advantage prior algorithms (applicable to density control of diffusions) do not have. The algorithm is demonstrated in controlling a nonlinear stochastic process using feedforward controls as well as (linear) state-parameterized control. We show the IDMP-DPP relationship for the corresponding density control problem.

\section{Appendix}
\label{Appendix}
%\noindent We provide the conditions for the existence of the adjoint operator defined here by through the following theorem.

\begin{proof} \textit{of theorem \ref{Thm1}}
	\normalfont
	\noindent We permit abuse of notation in this proof by neglecting to write argument dependencies o functions for brevity when necessary in this proof. Further we denote partial derivatives as $\frac{\partial f}{\partial \vv} = f_\vv$ if needed. This proof is a stepwise analogous extension to the multidimensional case of the proof for the one dimensional case (pp 199-203, Theorem 7.5) \cite{Hanson07appliedstochastic}. It follows easily by differentiating the well known multidimensional Dynkin's formula (pp 203, Equation 7.32) \cite{Hanson07appliedstochastic} for the function $v$ and using two integration by parts steps to move the spatial derivatives operating on $v$ to $\rp$. Differentiation of the Dynkin's formula for $u(t,\vx) = \Eb[v(\vx_t)|\vx_{t_0} = \vx]$ yields
	\begin{equation}
	\frac{\partial}{\partial t}u(t,\vx) = \Eb\left[\frac{\partial}{\partial t}\int\limits_{t_0}^{t}\mathcal{F}^{\dagger \; \vu}_\text{\tiny MJD} v (\vx_s)\rd \vs|\vx_{t_0} = \vx\right] = \int\limits_{\Rb^{n_x}} \mathcal{F}^{\dagger \; \vu}_\text{\tiny MJD} v (\vx) \rp(t,\vx) \rd \vx \label{PF1E1}
	\end{equation}
	\noindent where $\mathcal{F}^{\dagger \; \vu}_\text{\tiny MJD}$ is the backward operator defined in the previous subsection \ref{subsec:Problem_Formulation_Prelminaries}. Note that $u(t,\vx) = \Eb[v(\vx_t)|\vx_{t_0} = \vx]$ implies that when using the PDF representation of the expectation we have
	\begin{equation}
	\frac{\partial}{\partial t}u(t,\vx) = \int_{\Rb^{n_x}} v(\vx) \frac{\partial}{\partial t}\rp(t,\vx) \rd \vx. \label{PF1E2}
	\end{equation}
	\noindent The Dynkin formula \eqref{PF1E1} and the definition of the backward operator imply
	\begin{align}
	\frac{\partial}{\partial t} u(t,\vx) = \int\limits_{\Rb^{n_x}} \left(\sum\limits_{i = 1}^{n_x} F_i v_{x_i} + \frac{1}{2}\sum\limits_{i,j = 1}^{n_x} \Sigma_{ij} v_{x_j x_i} + \rd_\text{jump} v \right) \rp \rd \vx. \label{PF1E3}
	\end{align}
	\noindent Let us at first focus on the diffusion terms without the jump term in the above expression \eqref{PF1E3}. We use integration by parts for integration w.r.t. $x_i$ in step one, and $x_j$ in step two for terms from the backward operator in \eqref{PF1E1} to move the spatial derivatives to the PDF using condition (T2). Using this idea along with Fubini's theorem \cite{ThomasFinney_1996} we get \small
	\begin{align}
	&\int\limits_{\Rb^{n_x}} \sum\limits_{i = 1}^{n_x} \left( F_i v_{x_i} + \frac{1}{2}\sum\limits_{j = 1}^{n_x} \Sigma_{ij} v_{x_j x_i} \right)\rp \rd \vx = \underbrace{\int\limits_{\Rb} ... \int\limits_{\Rb}}_{(n_x - 1) \text{times}} \Bigg[ \sum\limits_{i = 1}^{n_x} \bigg\{ \int\limits_{\Rb} \bigg(  -\frac{\partial (F_i \rp)}{\partial x_i}\; v - \frac{1}{2} \sum\limits_{j = 1}^{n_x} \frac{\partial (\Sigma_{ij} \rp)}{\partial x_i} \; v_{x_j} \bigg) \rd x_i \notag \\
	&\qquad \qquad \qquad \qquad \qquad \qquad \qquad + \int\limits_{\Rb} \frac{\partial}{\partial x_i} \bigg(F_i \rp v + \frac{1}{2} \sum\limits_{j = 1}^{n_x} \Sigma_{ij} \rp v_{x_j} \bigg) \rd x_i \bigg\} \Bigg]  \notag \rd x_1 ... \rd x_{i-1}\rd x_{i+1} ... \rd x_{n_x} \notag
	\end{align}
	\begin{align}
	&= \int\limits_{\Rb^{n_x}} \sum\limits_{i = 1}^{n_x} -\frac{\partial (F_i \rp)}{\partial x_i} \; v \; \rd \vx + \int\limits_{\Rb^{n_x}} \frac{\partial}{\partial x_i} \bigg(F_i \rp v + \frac{1}{2} \sum\limits_{j = 1}^{n_x} \Sigma_{ij} \rp v_{x_j} \bigg) \rd \vx + \underbrace{\int\limits_{\Rb} ... \int\limits_{\Rb}}_{(n_x - 1) \text{times}} \Bigg[  \frac{1}{2} \sum\limits_{i,j = 1}^{n_x}  \int\limits_{\Rb} \bigg( \frac{\partial^2}{\partial x_i x_j} (\Sigma_{ij} \rp) v \bigg) \rd x_j \notag \\
	&\qquad \qquad \qquad \quad - \frac{1}{2} \sum\limits_{i,j = 1}^{n_x} \int\limits_{\Rb} \bigg(\frac{\partial}{\partial x_j} \bigg( \frac{\partial (\Sigma_{ij} \rp)}{\partial x_i} v \bigg) \bigg) \rd x_j \Bigg]  \notag \rd x_1 ... \rd x_{j-1}\rd x_{j+1} ... \rd x_{n_x} \notag \\
	&= \sum\limits_{i,j = 1}^{n_x} \int\limits_{\Rb^{n_x}} -\frac{\partial (F_i \rp)}{\partial x_i} \; v + \frac{1}{2}  \bigg( \frac{\partial^2 (\Sigma_{ij} \rp)}{\partial x_i x_j} \; v \bigg) \rd \vx + \sum\limits_{i,j = 1}^{n_x} \int\limits_{\Rb^{n_x}} \frac{\partial}{\partial x_i} \bigg(F_i \rp v + \frac{1}{2} \Sigma_{ij} \rp v_{x_j} - \frac{1}{2} \frac{\partial (\Sigma_{ij} \rp)}{\partial x_j} \; v \bigg) \rd \vx, \label{PF1E4}
	\end{align} \normalsize
	\noindent the last part of which can easily be identified as the conjunct in condition (T1). Let us now focus on the jump term in the expression \eqref{PF1E3} which is
	\begin{equation}
	\int\limits_{\Rb^{n_x}}\rd_\text{jump} v(\vx) \rp(t,\vx) \rd \vx = \int\limits_{\Rb^{n_x}} \sum\limits_{j = 1}^{n_p} \int\limits_{D_{Q_j}} \bigg( \Big[v(\vx + \vh_{j}(t,\vx, q_j)) - v(\vx) \Big] \rp_{Q_j} \lambda_j (t, q_j; t, \vx) \rd q_j \bigg) \rp(t,\vx) \rd \vx. \label{PF1E5}
	\end{equation}
	Consider the terms in the first summation on the right hand side of this equation. Change the variable of integration to $\vxi_j$ using the Change of Variables theorem \cite{JeffreysBook_1988} where $\vxi_j = \vx + \vh_{j}(t,\vx, q_j) = \vx + \veta_{j}(t,\vxi_j,q_j)$ in the first step, wherein $\vh_{j}$ is assumed invertible w.r.t. $\vxi_j$. This inverse mapping exists since we assumed $\vh_{j}$ is a bijection from $\Rb^{n_x}$ to $\Rb^{n_x}$. We note that transformed domain of integration is again $\Rb^{n_x}$. In the second step we change the dummy variable of integration back to $\vx$. We therefore have for all $j, \; 1 \leq j \leq n_p$ that
	\begin{align}
	&\int\limits_{\Rb^{n_x}} \int\limits_{D_{Q_j}} \bigg( v(\vx + \vh_{j}(t,\vx,q_j)) (p_{Q_j} \lambda_j)(t,q_j;t,\vx) \rd q_j \bigg) \rp(t,\vx)  \rd \vx \notag \\
	&= \int\limits_{\Rb^{n_x}} \int\limits_{D_{Q_j}} \bigg( v(\vxi_j) (p_{Q_j} \lambda_j)(t,q_j;t,\vxi_j - \veta_j(t,\vxi_j,q_j)) \rp(t,\vxi_j - \veta_j(t,\vxi_j,q_j)) \bigg) \rd q_j  |I - \veta_{j \vxi_j}(t,\vxi_j,q_j)| \rd \vxi_j \notag \\
	&= \int\limits_{\Rb^{n_x}} \int\limits_{D_{Q_j}} \bigg( v(\vx) \rp(t,\vx - \veta_j(t,\vx,q_j)) (p_{Q_j} \lambda_j)(t,q_j;t,\vx - \veta_j(t,\vx,q_j)) \bigg) \rd q_j|I - \veta_{j \vx}(t,\vx,q_j)| \rd \vx. \label{PF1E6}
	\end{align}
	\noindent Equations \eqref{PF1E3}, \eqref{PF1E4}, \eqref{PF1E5}, \eqref{PF1E6} and (T1) then imply
	\begin{align}
	&\frac{\partial}{\partial t} u(t,\vx) = \int\limits_{\Rb^{n_x}} \big(\sum\limits_{i = 1}^{n_x} F_i v_{x_i} + \sum\limits_{i,j = 1}^{n_x} \Sigma_{ij} v_{x_j x_i} + \rd_\text{jump} v \big) \rp \; \rd \vx= \sum\limits_{i,j = 1}^{n_x} \int\limits_{\Rb^{n_x}} \bigg[ -\frac{\partial}{\partial x_i}(F_i \rp) v + \frac{1}{2}  \frac{\partial^2}{\partial x_i x_j} (\Sigma_{ij} \rp) v \bigg] \rd \vx \notag \\
	&+ \sum\limits_{j = 1}^{n_p} \int\limits_{\Rb^{n_x}} \bigg[ \int\limits_{D_{Q_j}} \bigg( \rp(t,\vx - \eta_{j})  |I - \eta_{j \vx}| - \rp(t,\vx) \bigg) (p_{Q_j} \lambda_j)(t,q_j;t,\vx - \veta_j(t,\vx,q_j)) \bigg) \rd q_j \bigg] v(\vx) \rd \vx. \label{PF1E7}
	\end{align}
	\noindent Equations \eqref{PF1E2}, \eqref{PF1E7} and definition of forward Chapman-Kolmogorov operator in subsection \ref{subsec:Problem_Formulation_Prelminaries} imply \small
	\begin{align}
	&\frac{\partial}{\partial t} u(t,\vx) = \int\limits_{\Rb^{n_x}} \bigg[\sum\limits_{i,j = 1}^{n_x}  \bigg( -\frac{\partial}{\partial x_i}(F_i \rp) + \frac{1}{2}  \frac{\partial^2}{\partial x_i x_j} (\Sigma_{ij} \rp) \notag \\
	&+ \sum\limits_{j = 1}^{n_p} \int\limits_{D_{Q_j}} \Big( \rp(\vx - \eta_{j}(t,\vx))  |I - \eta_{j \vx}(t, \vx)| - \rp(t,\vx) \Big) (p_{Q_j} \lambda_j)(t,q_j;t,\vx - \veta_j(t,\vx,q_j)) \rd q_j \bigg] v(\vx) \rd \vx \notag \\
	&= \int\limits_{\Rb^{n_x}} \mathcal{F}^{\vu}_\text{\tiny MJD} \; \rp(t,\vx) \; v(\vx) \rd \vx= \int_{\Rb^{n_x}} \frac{\partial}{\partial t} \rp(t,\vx) v(\vx)  \rd \vx. \label{PF1E8}
	\end{align} \normalsize
	\noindent Using the calculus of variations argument (pp 201, proof of Theorem 7.5) \cite{Hanson07appliedstochastic}, \cite{Kirk1970} since the function $v$ is any arbitrary function with assumed boundedness and smoothness properties, $\rp(t,\vx)$ satisfies equation \eqref{Forward_CKPDE} in the weak sense.
%	\begin{equation}
%	\frac{\partial}{\partial t} \rp(t,\vx) = \mathcal{F}^{\vu}_\text{\tiny MJD} \; \rp(t,\vx).
%	\end{equation}
	\noindent Notice that the Dynkin formula \eqref{PF1E1} and equation \eqref{PF1E8} imply \small
	\begin{align}
	\frac{\partial}{\partial t} u(t,\vx) = \bigg< \mathcal{F}^{\dagger \; \vu}_\text{\tiny MJD} v (\vx), \rp(t,\vx) \bigg> = \bigg< \mathcal{F}^{\vu}_\text{\tiny MJD} \; \rp(t,\vx), v(\vx) \bigg>
	\end{align} \normalsize
	\noindent which means that the backward operator $\mathcal{F}^{\dagger \; \vu}_\text{\tiny MJD}$  is the formal adjoint operator of the forward operator $\mathcal{F}^{\vu}_\text{\tiny MJD}$. \\
	\noindent The delta initial condition to be proved is well known in the case that the jump term is absent, that is for the diffusion processes. However Poisson process $\vP_t$ undergoes jumps which causes $\vx_t$ to have discontinuous paths. However, considering that  simple jump process has Poisson distribution, we see that a jump is unlikely in a small time interval $\rd t$ from $\Pb(\rd \vP_t = 0) = \exp^{-\lambda(t) \rd t} \approxeq 1$ as $\rd t \rightarrow 0$ proving equation \eqref{PDF_canbe_delta}.
%	\begin{equation}
%	\lim\limits_{t \downarrow t_0}\rp(t,\vx) = \delta(\vx - \vx_0).
%	\end{equation}
\end{proof}

%\noindent We provide the conditions for the existence of the adjoint operator defined here by through the following theorem.
	
\begin{proof} \textit{of theorem \ref{Thm2}} 	\normalfont We use the process of liberation as in \cite{Lanczos1961} to liberate $\rp$ and obtain $\mathcal{F}^{\dagger \; \vu}_{\text{ \tiny MJD}}$. More precisely we start with the term $\pi (t,\vx) \mathcal{F}^{\vu}_{\text{\tiny MJD}} \rp (t,\vx)$ and manipulate it as follows
	\begin{align}
	\pi& (t,\vx) \mathcal{F}^{\vu}_{\text{\tiny MJD}} \rp (t,\vx) = \sum\limits_{i=1}^{n_x} \left( -\pi(\vx,t) \frac{\partial }{\partial x_i} ( F_i(t,\vx,\vu) \rp(t,\vx)) \right) + \frac{1}{2} \sum\limits_{i,j=1}^{n_x} \pi(t,\vx) \frac{\partial^2}{\partial x_i  \partial x_j}([\vSigma(t,\vx,\vu)]_{ij} \rp(\vx,t))  \notag
	\end{align}
	
	\begin{align}
	&+ \pi(t,\vx) \bigg(\sum\limits_{j = 1}^{n_p} \int\limits_{D_{Q_j}} \Big( \rp(\vx - \eta_{j}(t,\vx))  |I - \eta_{j \vx}(t, \vx)| - \rp(t,\vx) \Big) (p_{Q_j} \lambda_j)(t,q_j;t,\vx - \veta_j(t,\vx,q_j)) \rd q_j \bigg) \notag
	\end{align}
	 
	\begin{align}
	=& \sum\limits_{i = 1}^{n_x} \Big[ -\frac{\partial }{\partial x_i} \bigg( \pi(\vx,t) F_i(t,\vx,\vu) \rp(t,\vx)\bigg) + \rp(t,\vx) F_i(t,\vx,\vu(t)) \frac{\partial \pi(t,\vx)}{\partial x_i} \Big]\notag
	\end{align}
	\vspace{-3mm}
	\begin{align}
	&+ \frac{1}{2}   \sum\limits_{i,j = 1}^{n_x}   \bigg[ \frac{\partial}{\partial x_i} \Big( \pi(t,\vx)  \frac{\partial }{\partial x_j}([\vSigma(t,\vx,\vu)]_{ij} \rp(t,\vx) ) \Big) - \frac{\partial}{\partial x_i} \Big( [\vSigma(t,\vx,\vu)]_{ij} \rp(t,\vx) \frac{\partial \pi (t,\vx) }{\partial x_j}  \Big)
	\end{align}
	
	\begin{align}
	&{ +  \rp(t,\vx)  [\vSigma(t,\vx,\vu)]_{ij} \frac{\partial^2 \pi(t,\vx)}{\partial x_i \partial x_j}  \bigg] - \sum\limits_{j = 1}^{n_p} \int\limits_{D_{Q_j}} \pi(t,\vx) \rp(t,\vx)(p_{Q_j} \lambda_j)(t,q_j;t,\vx - \veta_j(t,\vx,q_j)) } \rd q_j \notag \\
	&+ \sum\limits_{j = 1}^{n_p} \int\limits_{D_{Q_j}} \pi(t,\vx) \rp(t,\vx - \veta_j) |I - \veta_{j \vx}| (p_{Q_j} \lambda_j)(t,q_j;t,\vx - \veta_j(t,\vx,q_j)) \rd q_j. \label{4.3}
	\end{align}	
	\noindent Integrating equation \eqref{4.3} over $\Rb^{n_x}$ and using (T4) gives us
	\begin{align}
	&\bigg<\pi, \mathcal{F}^{\vu}_{\text{\tiny MJD}} \rp \bigg> = \int\limits_{\Rb^{n_x}}  \bigg(  \sum\limits_{i=1}^{n_x}    \rp(t,\vx) F_i(t,\vx,\vu) \frac{\partial \pi}{\partial x_i} + \frac{1}{2} \sum\limits_{i,j = 1}^{n_x} \rp(t,\vx) [\vSigma(t,\vx,\vu)]_{ij} \frac{\partial^2 \pi}{\partial x_i    \partial x_j} \notag \\
	& + \sum\limits_{j = 1}^{n_p} \int\limits_{D_{Q_j}} \pi(t,\vx) \bigg((\rp(t,\vx - \veta_j)|I - \veta_{j \vx}|  - \rp(t,\vx))  (p_{Q_j} \lambda_j)(t,q_j;t,\vx - \veta_j(t,\vx,q_j)) \rd q_j   \bigg) \bigg)\rd\vx. \label{4.7}
	\end{align}
	\noindent Considering the terms in the last summation of this equation, we change the dummy variable of integration from $\vx$ to $\vxi_j$ in the first step. Then we choose $\vxi_j = \vx + \vh_{j}(t,\vx,q_j) = \vx + \veta_j(t,\vxi_j,q_j)$ in the second step where $\vh_{j}$ is assumed to be invertible w.r.t. $\vxi_j$ and change the variable of integration to $\vx$ using the Change of Variables Theorem \cite{JeffreysBook_1988}. This inverse mapping exists since we assume $\vh_{j}$ to be a bijection from $\Rb^{n_x}$ to $\Rb^{n_x}$. Noting that transformed domain of integration is again $\Rb^{n_x}$, we have that for all $1 \leq j \leq n_p$
	\begin{align}
	&\int\limits_{\Rb^{n_x}} \int\limits_{D_{Q_j}} \pi(t,\vx) \rp(t,\vx - \veta_j(t,\vx,q_j))  |I - \veta_{j \vx}(t,\vx,q_j)| (p_{Q_j} \lambda_j)(t,q_j;t,\vx - \veta_j(t,\vx,q_j)) \rd q_j \rd \vx \notag \\
	&= \int\limits_{\Rb^{n_x}} \int\limits_{D_{Q_j}} \pi(t,\vxi_j) \rp(t,\vxi_j - \veta_j(t,\vxi_j,q_j)) |I - \veta_{j \vxi_j}(t,\vxi_j,q_j)| (p_{Q_j} \lambda_j)(t,q_j;t,\vxi_j - \veta_j(t,\vxi_j,q_j)) \rd q_j \rd \vx \notag \\
	&= \int\limits_{\Rb^{n_x}} \int\limits_{D_{Q_j}} \rp(t,\vx) \; \pi(t,\vx + \vh_{j}(t,\vx,q_j)) (p_{Q_j} \lambda_j)(t,q_j;t,\vx) \rd q_j \rd\vx. \label{4.9}
	\end{align}	
	\noindent  So that equations \eqref{4.7}, \eqref{4.9} complete the proof as follows
	%	\begin{align}
	%	&\bigg<\pi,\mathcal{F}^\vu_{\text{MJD}} \rp \bigg> = \int\limits_{\Rb^{n_x}} \bigg(  \sum\limits_{i= 1}^{n_x} \rp(t,\vx) F_i(t,\vx,\vu) \frac{\partial \pi (t,\vx)}{\partial x_i} \notag \\
	%	&\small{+ \frac{1}{2}  \sum\limits_{i,j= 1}^{n_x} \rp(t,\vx) [\vSigma(t,\vx)]_{ij} \frac{\partial^2 \pi (t,\vx)}{\partial x_i \partial  x_j} + \sum\limits_{j = 1}^{n_p} - \pi(t,\vx) \rp(t,\vx)\lambda_j(t)} \notag \\
	%	&+ \sum\limits_{j = 1}^{n_p} \pi(t,\vx + \vh_{j}(t,\vx))  \rp(t,\vx)  \lambda_j(t) \bigg) \; \rd \vx.  \label{4.8}
	%	\end{align}
	%	\noindent Rewriting equation \eqref{4.8} we have	
	\begin{align}
	&\bigg<\pi,\mathcal{F}^\vu_{\text{\tiny MJD}} \rp \bigg> = \int\limits_{\Rb^{n_x}} \Bigg[ \sum\limits_{i = 1}^{n_x} F_i(t,\vx,\vu) \frac{\partial \pi(t,\vx)}{\partial x_i} + \frac{1}{2}  \sum\limits_{i,j = 1}^{n_x} [\vSigma(t,\vx,\vu)]_{ij} \frac{\partial^2 \pi(t,\vx)}{\partial x_i \partial x_j}  \notag \\
	&+ \sum\limits_{j = 1}^{n_p} \int\limits_{D_{Q_j}} \bigg(\pi(t,\vx + \vh_{j}(t,\vx,q_j)) - \pi(t,\vx) \bigg) (p_{Q_j} \lambda_j)(t,q_j;t,\vx) \rd q_j \Bigg] \rp(t,\vx)  \; \rd\vx = \bigg<\rp,\mathcal{F}^{\dagger \; \vu}_{\text{MJD}} \pi\bigg>. \label{4.10}
	\end{align}	
%	\noindent which completes the proof. 
	
\end{proof}

\section*{Acknowledgements}

This research was supported under the ARO W911NF-16-1-0390 award.

\small
%\newpage
\bibliographystyle{ieeetr}
\bibliography{References}

\end{document}